\documentclass[a4paper,12pt]{article}
\usepackage{amsfonts,amsmath,amsthm,amssymb}
\usepackage{graphicx}
\usepackage{epstopdf}
\usepackage{subfigure}
\usepackage{booktabs}
\usepackage{xcolor}
\usepackage{hyperref}
\usepackage[left=1in,right=1in,top=1in,bottom=1in]{geometry}

\RequirePackage[capitalize,nameinlink]{cleveref}[0.19]

\crefname{section}{section}{sections}
\crefname{subsection}{subsection}{subsections}
\Crefname{section}{Section}{Sections}
\Crefname{subsection}{Subsection}{Subsections}

\Crefname{figure}{Figure}{Figures}

\crefformat{equation}{\textup{#2(#1)#3}}
\crefrangeformat{equation}{\textup{#3(#1)#4--#5(#2)#6}}
\crefmultiformat{equation}{\textup{#2(#1)#3}}{ and \textup{#2(#1)#3}}
{, \textup{#2(#1)#3}}{, and \textup{#2(#1)#3}}
\crefrangemultiformat{equation}{\textup{#3(#1)#4--#5(#2)#6}}%
{ and \textup{#3(#1)#4--#5(#2)#6}}{, \textup{#3(#1)#4--#5(#2)#6}}{, and \textup{#3(#1)#4--#5(#2)#6}}

\Crefformat{equation}{#2Equation~\textup{(#1)}#3}
\Crefrangeformat{equation}{Equations~\textup{#3(#1)#4--#5(#2)#6}}
\Crefmultiformat{equation}{Equations~\textup{#2(#1)#3}}{ and \textup{#2(#1)#3}}
{, \textup{#2(#1)#3}}{, and \textup{#2(#1)#3}}
\Crefrangemultiformat{equation}{Equations~\textup{#3(#1)#4--#5(#2)#6}}%
{ and \textup{#3(#1)#4--#5(#2)#6}}{, \textup{#3(#1)#4--#5(#2)#6}}{, and \textup{#3(#1)#4--#5(#2)#6}}

\crefdefaultlabelformat{#2\textup{#1}#3}

\newcommand{\PI}{$\mathrm{P_I}$}
\newcommand{\PII}{$\mathrm{P_{II}}$}
\newcommand{\PVI}{$\mathrm{P_{VI}}$}
\newcommand{\Ai}{\mathrm {Ai}}
\newcommand{\Bi}{\mathrm {Bi}}
\renewcommand{\Re} {\mathop{\rm Re}\nolimits}
\renewcommand{\Im} {\mathop{\rm Im}\nolimits}

\newtheorem{theorem}{Theorem}[section]
\newtheorem{lemma}[theorem]{Lemma}
\newtheorem{corollary}[theorem]{Corollary}
\newtheorem{remark}[theorem]{Remark}
\numberwithin{equation}{section}

\title{Connection problem of the first Painlev\'{e} transcendent between poles and negative infinity
\thanks{
{This work was supported by
the National Natural Science Foundation of China under Grant nos. 11801480 and 12071394,
the Natural Science Foundation of Hunan Province under Grant no. 2020JJ5152,
the General Project of Hunan Provincial Department of Education under Grant no. 19C0771,
and the Doctoral Startup Fund of Hunan University of Science and Technology under Grant no. E51871.}}}

\author{Wen-Gao Long\thanks{School of Mathematics and Computation Science, Hunan University of Science and Technology, Xiangtan,  411201, China
  (\url{longwg@hnust.edu.cn}).}
\and Yu-Tian Li\thanks{School of Science and Engineering, Chinese University of Hong Kong, Shenzhen, Guangdong, 518172, China
  (\url{liyutian@cuhk.edu.cn}).}
\and Qing-hai Wang\thanks{Department of Physics, National University of Singapore, 117551, Singapore
  (\url{qhwang@nus.edu.sg}).}}

\date{}

\begin{document}

\maketitle

\begin{abstract}
We consider a connection problem of the first Painlev\'{e} equation (\PI),
trying to connect the local behavior (Laurent series) near poles and the asymptotic behavior as the variable $t$ tends to the negative infinity for real \PI~functions.
We get a classification of the real \PI~functions in terms of $(p,H)$ so that they behave differently at the negative infinity,
where $p$ is the location of a pole and $H$ is the free parameter in the Laurent series.
Some limiting-form connection formulas of \PI~functions are obtained for large $H$.
Specifically, for the real tritronqu\'{e}e solution, the large-$n$ asymptotic formulas of $p_n$ and $H_n$ are obtained,
where $p_n$ is the $n$-th pole on the real line in the ascending order
and $H_n$ is the associated free parameter.
Our approach is based on the complex WKB method (also known as the method of uniform asymptotics) introduced
by Bassom, Clarkson, Law and McLeod in their study on the connection problem of
the second Painlev\'{e} transcendent [Arch.~Rational Mech.~Anal., 1998, pp.~241-271].
Several numerical simulations are carried out to verify our main results.
Meanwhile, we obtain the phase diagram of \PI~solutions in the $(p,H)$ plane,
which somewhat resembles the Brillouin zones in solid-state physics.
The asymptotic and numerical results obtained in this paper partially answer Clarkson's open question on the connection problem of the first Painlev\'{e} transcendent.
\end{abstract}


\textbf{MSC classification (2020):}
  34M40, 34M55, 34E05, 33C10

\section{Introduction}
\label{sec:introduction}

At the turn of the twentieth century, Painlev\'{e} and his colleagues studied nonlinear ordinary differential equations on the complex plane,
they were { particularly} interested in those equations whose solutions { have} no ``movable'' branch points --
{ currently} known as the \emph{Painlev\'{e} property}.
Here, ``movable'' means that the location of singularities depends on the initial conditions.
Painlev\'{e} and his colleagues { found out that} there are only six irreducible equations,
meaning that { their} solutions, in general, can not be represented by elementary or classical special functions.
This group of six nonlinear ODEs are known as Painlev\'{e} equations, usually denoted by \PI--\PVI.
A nice survey of the derivation and historical background of the Painlev\'{e} equations can be found in Ince~\cite{Ince}
or Iwasaki--Kimura--Shimomura--Yoshida~\cite{IKSY}.

Though Painlev\'e equations were introduced for purely { mathematical} interest,
they turn out to play essential roles in many branches of mathematics and mathematical physics.
People { found out that} many nonlinear problems can be solved in terms of solutions of Painlev\'e equations, for example, {  integrable systems}, nonlinear waves, random matrix theory, number theory, Ising model, quantum gravity, etc.~\cite{IKSY}.
The solutions of the Painlev\'{e} equations are also called \emph{Painlev\'{e} transcendents},
and they are considered to be ``the nonlinear special functions''
that generalize the arsenal of classical special functions (Airy, Bessel, parabolic cylinder, hypergeometric functions, etc.),
see~\cite{Clarkson2003,Clarkson2006,Clarkson2019}.

The nonlinearity of Painlev\'e equations makes them fundamental tools in a broad array of nonlinear sciences,
and it also makes their properties and analysis complicated, on the other hand.
Since there is no nice closed-form expressions for solutions of \PI--\PVI,
the asymptotic approximation becomes a valuable tool to understand their properties.
These asymptotic approximations are usually established when the variable tends to some specific values and,
therefore, are local and usually can be done with considerable effort.
How to get a global picture of the solutions, that is, how to connect the asymptotic behaviors at different local regions,
is a much more intricate question, especially for the first Painlev\'{e} equation.

\subsection{The local asymptotic behaviors of \texorpdfstring{\PI}{PI}}
In the current work, we concentrate on the first Painlev\'e equation (\PI)
\begin{equation}\label{PI equation}
\frac{d^{2}y}{dt^{2}}=6y^2+t,
\end{equation}
and we consider the problem of connecting local behaviors of \PI~solutions.
Let us first describe the local asymptotic formulas for \PI.
An application of the idea of dominant balance to \cref{PI equation} yields that there are two kinds of solutions, behaving respectively
\[y(t)\sim\sqrt{-\frac{t}{6}}\quad \text{and}\quad y(t)\sim-\sqrt{\frac{-t}{6}}\]
as $t\rightarrow-\infty$; see for example Bender and Orszag~\cite{Bender-Orszag}.
Holmes and  Spence~\cite{HS-1984} studied the boundary value problem for \PI~and
they showed that there are exactly three types of real solutions of \PI~equation on the negative real axis.
Several years later, Kapaev~\cite{AAKapaev-1988} obtained the asymptotic approximations of these solutions,
which involve one or two parameters in each type, and
these parameters were given in terms of the so-called \emph{Stokes multipliers}, $\{s_k;k=0,1,2,3,4\}$.
We first list these asymptotic formulas and a full expression of the Stokes { multipliers} will be given in \cref{subsec:mon}.
Precisely speaking, Kapaev~\cite{AAKapaev-1988} showed \PI~has three types of solutions
with the following approximation formulas.

\begin{enumerate}
\item [(A)] a two-parameter family of solutions (termed \emph{oscillating solutions}), oscillating about the parabola $y=-\sqrt{-t/6}$ and satisfying
\begin{equation}\label{eq-behavior-type-A}
y=-\sqrt{-\frac{t}{6}}+d\,(-t)^{-\frac{1}{8}}\cos{\left[24^{\frac{1}{4}}\left(\frac{4}{5}(-t)^{\frac{5}{4}}-\frac{5}{8}d^2 \log(-t)+\varphi\right)\right]}+\mathcal{O}\left (t^{-\frac{5}{8}}\right )
\end{equation}
as $t\rightarrow-\infty$, where
{     \begin{equation}\label{eq-parameter-d-theta}
  \left\{\begin{aligned}
  &24^{\frac{1}{4}}d^{2}=-\frac{1}{\pi}\log{(1-|s_{2}|^2)},\\
  &24^{\frac{1}{4}}\varphi=\arg{s_{2}}-24^{\frac{1}{4}}d^{2}\left(\frac{19}{8}\log{2}+\frac{5}{8}\log{3}\right)+\frac{3\pi}{4}-\arg\Gamma\left(-i\frac{24^{\frac{1}{4}}}{2}d^{2}\right);
\end{aligned}\right.
\end{equation}}
\item [(B)] a one-parameter family of solutions (termed \emph{separatrix solutions}), satisfying
    \begin{equation}\label{eq-behavior-type-B}
      y(t)=\sqrt{\frac{-t}{6}}{ +}\frac{h}{4\sqrt{\pi}}24^{-\frac{1}{8}}(-t)^{-\frac{1}{8}}\exp\left\{-\frac{4}{5}24^{\frac{1}{4}}(-t)^{\frac{5}{4}}\right\}+\mathcal{O}\left(|t|^{-\frac{5}{2}}\right)
    \end{equation}
    as $t\rightarrow-\infty$, where
    \begin{equation}\label{eq-parameter-h}
     h=s_{1}-s_{4};
    \end{equation}
\item [(C)] a two-parameter family of solutions (termed \emph{singular solutions}), having infinitely many double poles on the negative real axis and satisfying
    \begin{equation}\label{eq-behavior-type-C}
      \frac{1}{y(t)+\sqrt{{-t}/{6}}}\sim { \frac{\sqrt{6}}{3}}(-t)^{-\frac{1}{2}}\sin^{2}\left\{\frac{2}{5}24^{1/4}(-t)^{\frac{5}{4}}+\frac{5}{8}\rho\log(-t)+\sigma\right\}
    \end{equation}
  as $t\rightarrow-\infty$, where
    \begin{equation}\label{eq-parameter-rho-sigma}
    \left\{\begin{aligned}
    \rho&=\frac{1}{2\pi}\log(|s_{2}|^{2}-1),\\
    \sigma&=\frac{19}{8}\rho\log{2}+\frac{5}{8}\rho\log{3}+\frac{1}{2}\arg\Gamma\left(\frac{1}{2}-i\rho\right)-\frac{\pi}{4}+\frac{1}{2}\arg{s_{2}}.
    \end{aligned}\right.
    \end{equation}
\end{enumerate}
The above three types solutions are classified by the Stokes multipliers $s_{k}$, $k=0,1,2,3,4$ as follows:
{ \begin{equation}\label{eq-condition-stokes-multipliers}
\begin{aligned}
\Im s_{0}=1-|s_{2}|^2>0 \quad \text{for Type (A)},\\
\Im s_{0}=1-|s_{2}|^2=0 \quad \text{for Type (B)},\\
\Im s_{0}=1-|s_{2}|^2<0 \quad \text{for Type (C)}.
\end{aligned}
\end{equation}
}

It is known that \PI~solutions are all meromorphic
-- every \PI~function has infinitely many poles on the complex plane,
and all the poles are of order two, i.e, double poles; { see \cite[pp.5-6]{GLS}}.
Suppose $p$ is a pole of the \PI~function $y(t)$,
then the Laurent series of $y(t)$ at $p$ takes the form
\begin{equation}\label{eq-Laurent-series}
y(t)=\frac{1}{(t-p)^2}-\frac{p}{10}(t-p)^2-\frac{1}{6}(t-p)^3+H(t-p)^4+\cdots,
\end{equation}
where the coefficient $H$ is a free parameter, and all later coefficients in the series depend only on $p$ and $H$.
Therefore, the pair $(p,H)$ fully characterizes the \PI~solutions $y(t)$.
In other words, $(p,H)$ depends only on the Cauchy data or the Stokes multipliers corresponding to $y(t)$.
Moreover, from \cite[Eqs. (34),~(36)]{Xia-Xu-Zhao}, one may note that
\begin{equation}\label{eq-relation-H-tau}
-28H=\lim\limits_{t\to p}2\left(\frac{d}{dt}\log{\tau(t)}-\frac{1}{t-p}\right),
\end{equation}
{ where $\tau(t)$ is the tau-function of \PI, which is defined by
$$\frac{d}{dt}\log{\tau(t)}=\frac{1}{2}\left(\frac{dy}{dt}\right)^2-2y^3-ty.$$
}
\subsection{Connection problems of \texorpdfstring{\PI}{PI}}

The asymptotic approximations of Painlev\'{e} transcendents are derived and valid { when} the variable tends to specific values.
For example, in the case of \PI, we have asymptotic approximations when $t$ approaches to one of these values:
\begin{enumerate}
  \item [(i)] when $t\to-\infty$, the solutions are described in \cref{eq-behavior-type-A}--\cref{eq-parameter-rho-sigma};
  \item [(ii)] when $t$ approaches to a pole $p$, the solutions are described in \cref{eq-Laurent-series};
  \item [(iii)] when $t$ is near the origin, the solutions are described by a Taylor series involving the initial data (Cauchy data).
\end{enumerate}
{ The} asymptotic formulas { appearing} in cases (i)-(iii) are valid \emph{locally} --
for a given solution of \PI, its behavior can be characterized in these local regions.
A natural and important question is how to connect these local behaviors
-- or put in the other way -- given a \PI~solution $y(t)$, the parameters in the approximations are fixed;
the question is how (if possible) to find the relations in explicit forms between the parameters used in different local regions.
If such a relation can be derived, it is usually called a \emph{connection formula},
and the question of building such formulas is known as the \emph{connection problem}.
We encounter with similar situations for linear differential equations on the complex plane,
and the connection problem can be solved for a large class of linear equations.
However, unlike the linear case,
Painlev\'e transcendents are more complicated and their connection problems have been studied and solved
in some interesting, albeit limited, cases.
Especially for the \PI~transcendents, the connection problems are widely open;
P.~A.~Clarkson announced finding the connection formulas for \PI~as open problems on several occasions \cite{Clarkson2003,Clarkson2006,Clarkson2019}.
Among these problems, we are particular interested in the following two connection problems.
\begin{enumerate}
\item [(1)] How to connect the asymptotic behaviors in (i) and (iii)?
That is to find the relation between parameters in \cref{eq-behavior-type-A}, \cref{eq-behavior-type-B} and \cref{eq-behavior-type-C} and the initial data $y(0)$ and $y'(0)$;

\item [(2)] How to connect the asymptotic behaviors in (i) and (ii)?
That is to find the relation between parameters used in \cref{eq-behavior-type-A}, \cref{eq-behavior-type-B} and \cref{eq-behavior-type-C}
and the ones in \cref{eq-Laurent-series}.
\end{enumerate}

The first problem is the initial value problem (Cauchy problem),
which was first considered by Holmes and Spence~\cite{HS-1984}.
They showed that there exist two constants $\kappa_{1}<0<\kappa_{2}$ such that all solutions of \cref{PI equation}
with $y(0)=0$ and $\kappa_{1}<y'(0)<\kappa_{2}$ belong to Type (A);
while $y'(0)>\kappa_{2}$ or $y'(0)<\kappa_{1}$, the solutions will blow up on the negative real axis.
Later on, several numerical investigations \cite{Bender-Komijani-2015, Fornberg-Weideman, Qin-Lu-2008}
revealed some interesting phenomena on the Cauchy problem of \PI.
In particular, Bender and Komijani~\cite{Bender-Komijani-2015} observed that the three types of \PI~solutions
appear alternatively as one initial data fixed and the other varying continuously.
Recently, Long et al.~\cite{LongLLZ} gave a rigorous proof to this phenomenon,
obtained an asymptotic classification of the \PI~solutions with respect to the initial data,
and built some limiting-form connection formulas.

The second problem is also natural since every real solution of \PI~has infinitely many poles on the real axis.
To understand the real solutions,
we need to know how the behaviors near poles connect to the behaviors at the negative infinity.
A complete solution to this connection problem is again tricky.
Nevertheless, the classification of the three types of behaviors of \PI~associated with $p$ and $H$
seems possible and deserves an investigation.

Inspired by the ideas in \cite{Sibuya-1967} and \cite{LongLLZ},
we find that the essential work is to approximate the Stokes multipliers for large $p$ or $H$,
which is the primary work in the present paper.
Precisely, we derive the asymptotic behavior of the Stokes multipliers $s_{k}$'s with large $H$,
while the location of poles $p$ can be arbitrary -- small, bounded fixed or large,
namely, there are three cases to be analyzed:
\[
\text{(i) } p\to 0;\qquad\quad
\text{(ii) } p \text{ is fixed};\qquad\quad
\text{(iii) } p\to\infty.
\]
For each of these cases,
we will classify the \PI~solutions in terms of $p$ and $H$,
and build the corresponding limiting-form connection formulas.

Every \PI~transcendent has infinitely many of poles on the complex plane, and the poles can be anywhere in general.
Among all solutions of \PI, there are two special families, called tronqu\'{e}e and tritronqu\'{e}e solutions.
If the complex plane is divided into five equal sectors, then tronqu\'{e}e solutions are pole-free on two
and tritronqu\'{e}e solutions are pole-free on four sectors, respectively.
There is basically only one tritronqu\'{e}e solution and the other four can be obtained by a simple rotation of the variable.
The distribution of poles of these solutions has been { receiving} a lot of attention in research,
see~\cite{Bertola-Tovbis-2013,Costin-Costin-Huang-2015,Costin-Costin-Huang-2016,Costin-Huang-Tanveer-2016,Dubrovin2009,Joshi-Kitaev2001,Masoero-2010}.
A by-product of our current study is that we find the precise asymptotic approximation of $(p,H)$ in the Laurent series of the $n$-th pole of the real tritronqu\'{e}e solution, see \cref{cor-tritronquee} below.

Just like the initial conditions \cite[Figure 4.5]{Fornberg-Weideman},
one may find regions in the parameter space $(p,H)$ with each region gives rise to a particular type of \PI~solutions.
We call such a classification map as ``a phase diagram,'' a term borrowed from physics.
However, unlike the initial conditions, we do not have a one-to-one correspondence
between a point in the $(p,H)$ space and a \PI~solution.
This is because all real \PI~solutions have infinite double poles on the real axis,
regardless of the type of the solution.
That is, there are infinite points in the parameter space $(p,H)$ that correspond to a single \PI~solution.
It is natural to divide the entire parameter space into infinite regions,
with the points in each region have a one-to-one correspondence with \PI~solutions.
The points in different such regions are connected by \PI~solutions. We will see that such a partition has certain arbitrariness.
In a sense, these regions resemble the Brillouin zones in solid state physics.

\subsection{Monodromy theory for \texorpdfstring{\PI~and}{PI and} RTHE}
\label{subsec:mon}

We recall some important concepts in the monodromy theory for the first Painlev\'{e} transcendents.
Recall that one of the Lax pairs for the \PI~equation is (see \cite{Kapaev-Kitaev-1993})
\begin{equation}\label{lax pair-I}
\left\{\begin{aligned}
\frac{\partial\Psi}{\partial\lambda}
&=\left\{(4\lambda^4+t+2y^2)\sigma_{3}-i(4y\lambda^2+t+2y^2)\sigma_{2}-(2y_{t}\lambda+\frac{1}{2\lambda})\sigma_{1}\right\}\Psi,
\\
\frac{\partial\Psi}{\partial t}
&=\left\{(\lambda+\frac{y}{\lambda})\sigma_{3}-\frac{iy}{\lambda}\sigma_{2}\right\}\Psi,
\end{aligned}\right.
\end{equation}
where
$$\sigma_{1}=\begin{bmatrix}0&1\\1&0\end{bmatrix},\quad \sigma_{2}
=\begin{bmatrix}0&-i\\i&0\end{bmatrix},\quad \sigma_{3}
=\begin{bmatrix}1&0\\0&-1\end{bmatrix}$$
are the Pauli matrices and $y_{t}=\frac{dy}{dt}$.
The compatibility of the above system means $\frac{\partial^2\Psi}{\partial\lambda\partial t}=\frac{\partial^2\Psi}{\partial t\partial\lambda}$,
which implies that $y=y(t)$ satisfies the first Painlev\'{e} equation \cref{PI equation}.
Under the transformation
\begin{equation}\label{eq-transform-canonical solution}
\Phi(\lambda)
=\lambda^{\frac{1}{4}\sigma_{3}}\frac{\sigma_{3}+\sigma_{1}}{\sqrt{2}}\Psi(\sqrt{\lambda}),
\end{equation}
the first equation of \cref{lax pair-I} becomes
\begin{equation}\label{eq-fold-Lax-pair}
\frac{\partial\Phi}{\partial\lambda}
=\begin{bmatrix}y_{t}&{ 2\lambda^{2}+2y\lambda+t+2y^2}\\2(\lambda-y)&-y_{t}\end{bmatrix}\Phi.
\end{equation}
The only singularity of the above equation is the irregular singular point at $\lambda=\infty$.
Following \cite{Kapaev-Kitaev-1993} (see also \cite{AAKapaev-2004}),
there exist canonical solutions $\Phi_{k}(\lambda)$, $k\in\mathbb{Z}$, of \cref{eq-fold-Lax-pair} with the asymptotic expansion
\begin{equation}\label{eq-canonical-solutions}
\Phi_{k}(\lambda,t)
=\lambda^{\frac{1}{4}\sigma_{3}}\frac{\sigma_{3}+\sigma_{1}}{\sqrt{2}}
\left({ I+\frac{\mathcal{H}}{\sqrt{\lambda}}+\mathcal{O}\left(\frac{1}{\lambda}\right)}\right)
e^{(\frac{4}{5}\lambda^{\frac{5}{2}}+t\lambda^{\frac{1}{2}})\sigma_{3}}
\end{equation}
as $\lambda\rightarrow\infty$ with $\lambda\in\Omega_{k}$, uniformly for all $t$ bounded away from $p$, where $\mathcal{H}=-(\frac{1}{2}y_{t}^2-2y^3-ty)\sigma_{3}$, and the canonical sectors are
$$\Omega_{k}=\left\{\lambda\in\mathbb{C}:~\arg \lambda\in \left(-\frac{3\pi}{5}+\frac{2k\pi}{5},\frac{\pi}{5}+\frac{2k\pi}{5}\right)\right\}, \qquad k\in\mathbb{Z}.$$
These canonical solutions are related by
\begin{equation}\label{eq-Stokes-matrices}
\Phi_{k+1}=\Phi_{k}S_{k},\quad S_{2k-1}=\begin{bmatrix}1&s_{2k-1}\\0&1\end{bmatrix},\quad S_{2k}=\begin{bmatrix}1&0\\s_{2k}&1\end{bmatrix},
\end{equation}
where $s_{k}$ are called \emph{Stokes multipliers},
and independent of $\lambda$ and $t$ according to the isomonodromy condition.
The Stokes multipliers are subject to the constraints
\begin{equation}\label{eq-constraints-stokes-multipliers}
s_{k+5}=s_{k}
\quad\text{and}\quad
s_{k}=i(1+s_{k+2}s_{k+3}),
\qquad
k\in\mathbb{Z}.
\end{equation}
Moreover, regarding $s_{k}$ as functions of $(t,y(t),y'(t))$,
they also satisfy
\begin{equation}\label{eq-sk-s-k-relation}
s_{k}\left (t,y(t),y'(t)\right)
=-\overline{s_{-k}\left (\bar{t},\overline{y(t)},\overline{y'(t)}\right )},
\qquad
k\in\mathbb{Z},
\end{equation}
where $\bar z$ stands for the complex conjugate of $z$, see~\cite[(13)]{AAKapaev-1988}.
It is readily seen from \cref{eq-constraints-stokes-multipliers} that,
in general, two of the Stokes multipliers determine all others.
According to \cite{Bertola-Tovbis-2013}, we can proceed further as follows.
Define
\begin{equation}\label{eq-def-hat{Phi}}
\hat{\Phi}(\lambda,t)=G(\lambda,t)\Phi(\lambda,t)
\end{equation}
with
\begin{equation}\label{def-G(lambda,t)}
G(\lambda,t)
=\begin{bmatrix}0&1\\ 1&{ -\frac{1}{2}\left(-y_{t}+\frac{1}{2(\lambda-y)}\right)}\end{bmatrix}
(\lambda-y)^{\frac{\sigma_{3}}{2}}.
\end{equation}
Then $\hat{\Phi}(\lambda,t)$ satisfies
\begin{equation}\label{eq-system-hat-Phi}
\frac{d}{d\lambda}\hat{\Phi}(\lambda,t)
=\begin{bmatrix}0&2\\V(\lambda,t)&0\end{bmatrix}\hat{\Phi}(\lambda,t),
\end{equation}
where
\[
2V(\lambda,t)=y_{t}^{2}+4\lambda^{3}+2\lambda t-2y t-4y^{3}-\frac{y_{t}}{\lambda-y}+\frac{3}{4}\frac{1}{(\lambda-y)^2}.
\]
Moreover, using the approximation of $\Phi(\lambda,t)$ in \cref{eq-canonical-solutions}, one can verify that
\begin{equation}\label{eq-canonical-solutions-hat-Phi}
\hat{\Phi}_{k}(\lambda,t)
=\frac{\lambda^{-\frac{3}{4}\sigma_{3}}}{\sqrt{2}}
\begin{bmatrix}1&-1\\1&1\end{bmatrix}
\left(I+\mathcal{O}(\lambda^{-\frac{1}{2}})\right)
e^{(\frac{4}{5}\lambda^{\frac{5}{2}}+t\lambda^{\frac{1}{2}})\sigma_{3}}
\end{equation}
as $\lambda\rightarrow\infty$ with $\lambda\in\Omega_{k}$.
The asymptotic expansions of $\Phi_{k}(\lambda,t)$ and $\hat{\Phi}_{k}(\lambda,t)$ in \cref{eq-canonical-solutions} and \cref{eq-canonical-solutions-hat-Phi} are valid
only when $t$ is bounded away from $p$.
When $t\to p$, according to \cite{Bertola-Tovbis-2013},
the above system \cref{eq-system-hat-Phi} turns to
\begin{equation}\label{eq-system-hat-Phi2}
\frac{d}{d\lambda}\hat{\Phi}(\lambda,p)
=\begin{bmatrix}0&2\\2\lambda^3+p\lambda-14H&0\end{bmatrix}\hat{\Phi}(\lambda,p),
\end{equation}
and the corresponding asymptotic expansions of $\hat{\Phi}_{k}(\lambda,p)$ in \cref{eq-canonical-solutions-hat-Phi} should be replaced by
\begin{equation}\label{eq-canonical-solutions-near-pole}
\hat{\Phi}_{k}(\lambda,p)
=\frac{\lambda^{-\frac{3}{4}\sigma_{3}}}{\sqrt{2}}
\begin{bmatrix}-i&-i\\-i&i\end{bmatrix}
\left(I+\mathcal{O}\left(\frac{1}{\sqrt{\lambda}}\right)\right)
e^{(\frac{4}{5}\lambda^{\frac{5}{2}}+p\lambda^{\frac{1}{2}})\sigma_{3}}
\end{equation}
as $\lambda\rightarrow\infty$ with $\lambda\in\Omega_{k}$; { see \cite[Corollary A.8]{Bertola-Tovbis-2013}}. Finally, if denoting
\[
\hat{\Phi}(\lambda,p)=\begin{bmatrix}\phi_{1}\\\phi_{2}\end{bmatrix},
\]
and letting $Y(\lambda;p)=\phi_{1}$, then we arrive at the { following reduced triconfluent Heun equation (RTHE); see  \cite[Eq. (6)]{Xia-Xu-Zhao} or \cite[p.108]{SL}}
\begin{equation}\label{Schrodinger-equation-triconfluent-Heun}
\frac{d^{2}Y}{d\lambda^{2}}=\left[4\lambda^{3}+2p\lambda-28H\right]Y.
\end{equation}

\begin{remark}
It is worth mentioning that $\Phi_{k}(\lambda,t)$ and $\hat{\Phi}_{k}(\lambda,p)$ share the same Stokes matrices,
which follows from the isomonodromic property of equation \cref{eq-fold-Lax-pair}
and the fact that left multiplying by $G(\lambda,t)$ does not change the Stokes phenomenon.
{ It should be noted that, according to \cite[Corollary A.8]{Bertola-Tovbis-2013}, the asymptotic behaviors of $\hat{\Phi}_{k}(\lambda,t)$ in the two cases $t\neq p$ and $t=p$ are different. The reason is that there is a term
$\left(\frac{\sqrt{\xi}+\sqrt{y}}{\sqrt{\xi-y}}\right)^{\sigma_{3}}$ in \cite[A.38]{Bertola-Tovbis-2013}. When $t$ is away from $p$, $|y|\ll |\xi|$, then $\left(\frac{\sqrt{\xi}+\sqrt{y}}{\sqrt{\xi-y}}\right)^{\sigma_{3}}$ is asymptotic to the identity matrix. Nevertheless, when $t=p$, we have $|y|\gg |\xi|$, hence $\left(\frac{\sqrt{\xi}+\sqrt{y}}{\sqrt{\xi-y}}\right)^{\sigma_{3}}\sim(-i)^{\sigma_{3}}$. We should also mention that $\Phi(\lambda,t)$ in this paper is equal to $\Psi(\xi; v)\begin{bmatrix}1&0\\0&i\end{bmatrix}$ in \cite{Bertola-Tovbis-2013},
and therefore there is a corresponding difference between \cref{eq-canonical-solutions-near-pole} and \cite[(4.17)]{Bertola-Tovbis-2013}. (Note the different notations between the present paper and \cite{Bertola-Tovbis-2013}. For instance, the symbol $\xi$ in \cite{Bertola-Tovbis-2013} is $\lambda$ in the present paper.)}
\end{remark}

The rest of this paper is organized as follows. In \cref{sec:main},
we state our main results in two theorems and several corollaries.
The proof of the corollaries is also contained in this section.
Then, the method of ``uniform asymptotics'' is carried out case by case in \cref{sec:proof-lemma} to prove our main theorems.
\Cref{sec:numerical} focuses on the numerical simulations which verify our main results, and meanwhile,
we obtain some new heuristic information of the first Painlev\'{e} functions.
Finally, a few concluding remarks are given in \cref{sec:discussion}.

\section{Main results}
\label{sec:main}

We first obtain the leading term asymptotics of the Stokes multipliers of
the reduced triconfluent Heun equation~\cref{Schrodinger-equation-triconfluent-Heun} for large $H$ and arbitrary $p$.
Then, as consequences of these asymptotic formulas and \cref{eq-condition-stokes-multipliers},
we give asymptotic classifications of the \PI~solutions in terms of $p$ and $H$. The main results are stated in two cases.
\ \\
\paragraph{$\bullet$ Case I: $H\to-\infty$}
\ \\

Suppose that $y(t; p, H)$ is the real solution of \cref{PI equation} with a pole at $t=p$ and $H$ being the free parameter in the Laurent series.
Set
\[
H=-\frac{\xi^{6/5}}{7}
\quad \text{and}\quad
p=2C(\xi)\xi^{4/5}
\]
with $\xi$ being a large real number, and define
\begin{equation}\label{eq-def-kappa-hat-kappa}
\kappa^2=\frac{4}{\pi i}\int_{\eta_{1}}^{\eta_{2}}\left(s^3+C(\xi)s+1\right)^{\frac{1}{2}}ds \quad \text{and}\quad \hat{\kappa}^2=\frac{4}{\pi i}\int_{\eta_{0}}^{\eta_{1}}\left(s^3+C(\xi)s+1\right)^{\frac{1}{2}}ds,
\end{equation}
where $C(\xi)\in\mathbb{R}$ and $\eta_{i}\,(i=0,1,2)$ are the roots of $\eta^3+C(\xi)\cdot \eta+1=0$ with $\Re\eta_{0}<0$.
The other two roots $\eta_{1}, \eta_{2}$ are either positive ($\Re \eta_{1}\leq\Re \eta_{2}$) or form a conjugate pair ($\Im \eta_{1}\geq \Im \eta_{2}$).
Here and in what follows, the branch cuts are chosen
so that $\arg(s-\eta_{0})\in\left(-\pi,\pi\right]$, $\arg(s-\eta_{1,2})\in\left[-\frac{\pi}{2},\frac{3\pi}{2}\right)$;
see \cref{branches} for the branch cuts and the integration paths.
\begin{lemma}\label{lem-def-C0}
Whenever $C(\xi)>-3/2^{2/3}$, we have $\Re\hat{\kappa}^2<0$. Moreover, there exists a unique $C_{0}>0$ such that
\begin{equation}\label{eq-kappa-hat-kappa-sign}
\kappa^2\begin{cases}
>0,\quad C(\xi)<-3/2^{2/3},\\
=0,\quad C(\xi)=-3/2^{2/3},\\
<0,\quad C(\xi)>-3/2^{2/3},
\end{cases}
\qquad
\Im \hat{\kappa}^2
\begin{cases}
<0, \quad C(\xi)>C_{0},\\
=0,\quad C(\xi)=C_{0},\\
>0,\quad C(\xi)<C_{0}.
\end{cases}
\end{equation}
\end{lemma}
The proof of this lemma is left in Appendix A. Numerical computation shows that the value of $C_{0}$ is closed to $2.004\,860\,503\,264\,124$. Now we state our results of case I: $H\to-\infty$ in the following theorem.

\begin{figure}
\centering
\subfigure{
\includegraphics[width=0.27\textwidth]{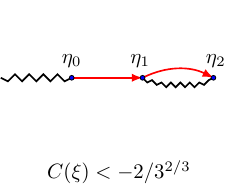}}
  \subfigure{
\includegraphics[width=0.27\textwidth]{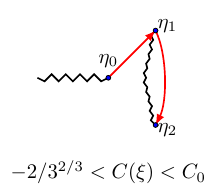}}
\subfigure{
\includegraphics[width=0.27\textwidth]{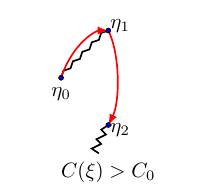}}
\caption{The branch cuts of the square roots and the integration contours in \cref{eq-def-kappa-hat-kappa}.}
\label{branches}
\end{figure}

\begin{theorem}\label{Thm-H-negative}
The asymptotic behaviors of the Stokes multipliers, corresponding to $y(t; p, H)$, are stated as follows.
\begin{enumerate}
\item [(i)] For any $\delta_{1}\in(0,C_{0})$, then as $\xi\to+\infty$
\begin{equation}\label{s-C-region1}
\begin{aligned}
s_{0}&=\left[-\frac{\sqrt{2\pi}i}{\Gamma\left(\frac{1}{2}+\frac{\xi\kappa^2}{2}\right)}+R_{0}(\xi)\right]2^{\frac{\xi\kappa^2}{2}}\xi^{\frac{\xi\kappa^2}{2}}e^{-2\xi E(\xi)}, \\
s_{1}&=\left[-\frac{\sqrt{2\pi} i}{\Gamma\left(\frac{1}{2}-\frac{\xi\kappa^2}{2}\right)}+R_{1}(\xi)\right]2^{-\frac{\xi\kappa^2}{2}}\xi^{-\frac{\xi\kappa^2}{2}}e^{2\xi E(\xi)+\frac{\xi\kappa^2\pi i}{2}}
\end{aligned}
\end{equation}
with
\begin{equation}\label{approx-bound-R0}
\begin{aligned}
R_{0}(\xi)&=\begin{cases}
 \mathcal{O}\left(\frac{\xi^{-1}}{\Gamma\left(\frac{1}{2}+\frac{\xi\kappa^2}{2}\right)}\right) ,  &\text{ when } \kappa^2>0,\\
  \mathcal{O}\left(\xi^{-1}\Gamma\left(\frac{1}{2}-\frac{\xi\kappa^2}{2}\right)\right) ,  &\text{ when } \kappa^2<0,\\
\end{cases}\\
R_{1}(\xi)&=\begin{cases}
 \mathcal{O}\left(\xi^{-1}\Gamma\left(\frac{1}{2}+\frac{\xi\kappa^2}{2}\right)\right) ,  &\text{ when } \kappa^2>0,\\
  \mathcal{O}\left(\frac{\xi^{-1}}{\Gamma\left(\frac{1}{2}-\frac{\xi\kappa^2}{2}\right)}\right) ,  &\text{ when } \kappa^2<0,\\
\end{cases}
\end{aligned}
\end{equation}
hold uniformly  for all $C(\xi)\in(-\infty,\delta_{1}]$;
\item [(ii)] For any $\delta_{2}\in(-3/2^{2/3},0)$, then as $\xi\to+\infty$
\begin{equation}\label{s-C-region2}
\begin{aligned}
s_{1}&=\left[-\frac{\sqrt{2\pi} i}{\Gamma\left(\frac{1}{2}-\frac{\xi\hat{\kappa}^2}{2}\right)}+\hat{R}_{1}(\xi)\right]2^{-\frac{\xi\hat{\kappa}^2}{2}}\xi^{-\frac{\xi\hat{\kappa}^2}{2}}e^{\frac{\xi\hat{\kappa}^2\pi i}{2}+2\xi F(\xi)},\\
s_{2}&=\left[\frac{-\sqrt{2\pi} i}{ \Gamma\left(\frac{1}{2}+\frac{\xi\hat{\kappa}^2}{2}\right)}+\hat{R}_{2}(\xi)\right]2^{\frac{\xi\hat{\kappa}^2}{2}}\xi^{\frac{\xi\hat{\kappa}^2}{2}}e^{-\xi\pi i \hat{\kappa}^2 -2\xi F(\xi)}
\end{aligned}
\end{equation}
with
\[
\hat{R}_{1}(\xi)=
  \mathcal{O}\left(\frac{\xi^{-1}}{\Gamma\left(\frac{1}{2}-\frac{\xi\hat{\kappa}^2}{2}\right)}\right),
\qquad
\hat{R}_{2}(\xi)=
\mathcal{O}\left(\xi^{-1}\Gamma\left(\frac{1}{2}-\frac{\xi\hat{\kappa}^2}{2}\right)\right),
\]
hold uniformly for all $C(\xi)\in[\delta_{2},+\infty)$.
\end{enumerate}
The explicit expressions of $E(\xi)$ and $F(\xi)$ are given by
\begin{equation*}
\begin{split}
E(\xi)&=I_{E}(\xi)+\frac{\kappa^2}{4}+\frac{\kappa^2\log{2}}{2}-\frac{\kappa^2\log{\kappa^2}}{4},\\
F(\xi)&=I_{F}(\xi)+\frac{\hat{\kappa}^2}{4}+\frac{\hat{\kappa}^2\log{2}}{2}-\frac{\hat{\kappa}^2\log{\hat{\kappa}^2}}{4}
\end{split}
\end{equation*}
and
\begin{equation} \label{eq-def-I-E(xi)}
\begin{aligned}
I_{E}(\xi)=2\int_{\eta_{2}}^{\infty e^{i\theta_{1}}}&\left[\left(s^3+C(\xi)s+1\right)^{\frac{1}{2}}-\left(s^{\frac{3}{2}}+\frac{C(\xi)}{2}s^{-\frac{1}{2}}\right)\right]ds\\
&-\left(\frac{4}{5}(\eta_{2})^{\frac{5}{2}}+2C(\xi)(\eta_{2})^{\frac{1}{2}}\right)\\
I_{F}(\xi)=2\int_{\eta_{1}}^{\infty e^{i\theta_{2}}}&\left[\left(s^3+C(\xi)s+1\right)^{\frac{1}{2}}-\left(s^{\frac{3}{2}}+\frac{C(\xi)}{2}s^{-\frac{1}{2}}\right)\right]ds\\
&-\left(\frac{4}{5}(\eta_{1})^{\frac{5}{2}}+2C(\xi)(\eta_{1})^{\frac{1}{2}}\right)
\end{aligned}
\end{equation}
with $\theta_{1}\in\left(-\pi,\frac{3\pi}{5}\right)$ and $\theta_{2}\in\left(-\frac{3\pi}{5},\pi\right)$. Moreover, we have $I_{F}(\xi)=\overline{I_{E}(\xi)}$ and
\begin{equation}\label{eq-relation-kappa-IE-IF}
\begin{aligned}
&\kappa^2\pi=\Im(2I_{F}(\xi)-2I_{E}(\xi))=4\Im{I_{F}(\xi)}=\Im{(2\hat{\kappa}^2\pi i)},\\
&\Re{(\hat{\kappa}^2\pi i+2 I_{F}(\xi))}=0
\end{aligned}
\end{equation}
when $-3/2^{2/3}\leq C(\xi)\leq C_{0}$.
\end{theorem}

\begin{remark}
In \cref{Thm-H-negative},
we state the leading-term asymptotics of the Stokes multipliers of the reduced triconfluent Heun equation with large negative $H$ and arbitrary $p$.
Although the results are stated in two regions for $C(\xi)$ respectively,
one can show that the leading behaviors of the Stokes multipliers in \cref{s-C-region1} and \cref{s-C-region2}
are consistent in the overlapping region $\delta_{2}\leq C(\xi)\leq\delta_{1}$.
Indeed, from \cref{s-C-region1}, we find that
\begin{equation}\label{eq-s0-approx-1}
s_{0}= -i e^{-2\xi I_{E}(\xi)+\frac{\xi\kappa^2\pi i}{2}}\cos\left[\frac{\xi\kappa^2\pi}{2}+\mathcal{O}(\xi^{-1})\right], \quad \text{as}\quad \xi\to+\infty.
\end{equation}
On the other hand, since $s_{k}=i(1+s_{k+2}s_{k+3}), k\in\mathbb{Z}$,
we find that $s_{0}=\frac{i-s_{2}}{1+s_{1}s_{2}}$.
Hence, from \cref{s-C-region2}, one can obtain that whenever $\delta_{2}\leq C(\xi)\leq\delta_{1}$
\begin{equation}\label{eq-s0-approx-2}
s_{0}\sim -i\left[e^{\xi\hat{\kappa}^2\pi i}+e^{-2\xi I_{F}(\xi)}\right], \quad \text{as}\quad \xi\to+\infty.
\end{equation}
According to the relations of $\kappa^2$, $\hat{\kappa}^2$ and $I_{E}(\xi),\,I_{F}(\xi)$ in~\cref{eq-relation-kappa-IE-IF},
we further conclude that $\Re{(\hat{\kappa}^2\pi i)}=-2\Re{I_{F}(\xi)}=-2\Re{I_{E}(\xi)}$ and $\Im{(\hat{\kappa}^2\pi i)}=\frac{\kappa^2\pi}{2}=2\Im{I_{F}(\xi)}$
whenever $\delta_{2}\leq C(\xi)\leq\delta_{1}$.
These facts imply that the leading asymptotic approximations of $s_{0}$ in \cref{s-C-region1} and \cref{s-C-region2} are consistent.
A similar analysis can be done for $s_{1}$.
In view of $s_{k}=i(1+s_{k+2}s_{k+3}), k\in\mathbb{Z}$,
it is straightforward to derive the asymptotic behaviors of other Stokes multipliers.
\end{remark}

\begin{corollary}\label{cor-alternate-p-H-large}
Let $C(\xi)=C$ be fixed, $p=2C\xi^{4/5}$, $H=-\frac{\xi^{6/5}}{7}$ and $\kappa^2$ be defined in \cref{eq-def-kappa-hat-kappa} with $C(\xi)$ replaced by $C$.
\begin{enumerate}
\item [(i)]{  If $C\in(-3/2^{2/3},C_{0})$, then there exists $M_{1}>0$ and a sequence $\{\xi_{n}\}$ with $M_{1}<\xi_{1}<\xi_{2}<\cdots<\xi_{n}<\cdots$ } such that the \PI~solutions $y(t; p, H)$, with $\xi=\xi_{n}$, belongs to Type (B) and satisfies the asymptotic behavior \cref{eq-behavior-type-B} with
\begin{equation}\label{eq-def-hn}
h:=h_{n}\sim(-1)^{n}\frac{\sqrt{\pi} }{\Gamma\left(n\right)}2^{n}\xi_{n}^{-\frac{1-2n}{2}}e^{2\xi_{n} E(\xi_{n})}, \quad\text{as}\quad n\to+\infty.
\end{equation}
Moreover, we have
\begin{equation}\label{eq-def-xi-n}
\xi_{n}\sim\frac{2n-1}{-\kappa^2}\quad \text{as}\quad n\to+\infty.
\end{equation}
Further more, if $\xi\in\left(\xi_{2m-1}, \xi_{2m}\right)$, $m=1,2,\cdots$,
then $y(t; p, H)$ belongs to Type (A).
Otherwise, when $\xi\in\left(\xi_{2m}, \xi_{2m+1}\right)$, $y(t; p, H)$ belongs to Type (C).
\item [(ii)] If $C<-3/2^{2/3}$ or $C>C_{0}$, then there exists $M_{2}>0$ such that for any $\xi>M_{2}$ the \PI~solutions $y(t; p, H)$ belong to Type (C).
\end{enumerate}
\end{corollary}

\begin{figure}
\centering
  \includegraphics[width=0.8\textwidth]{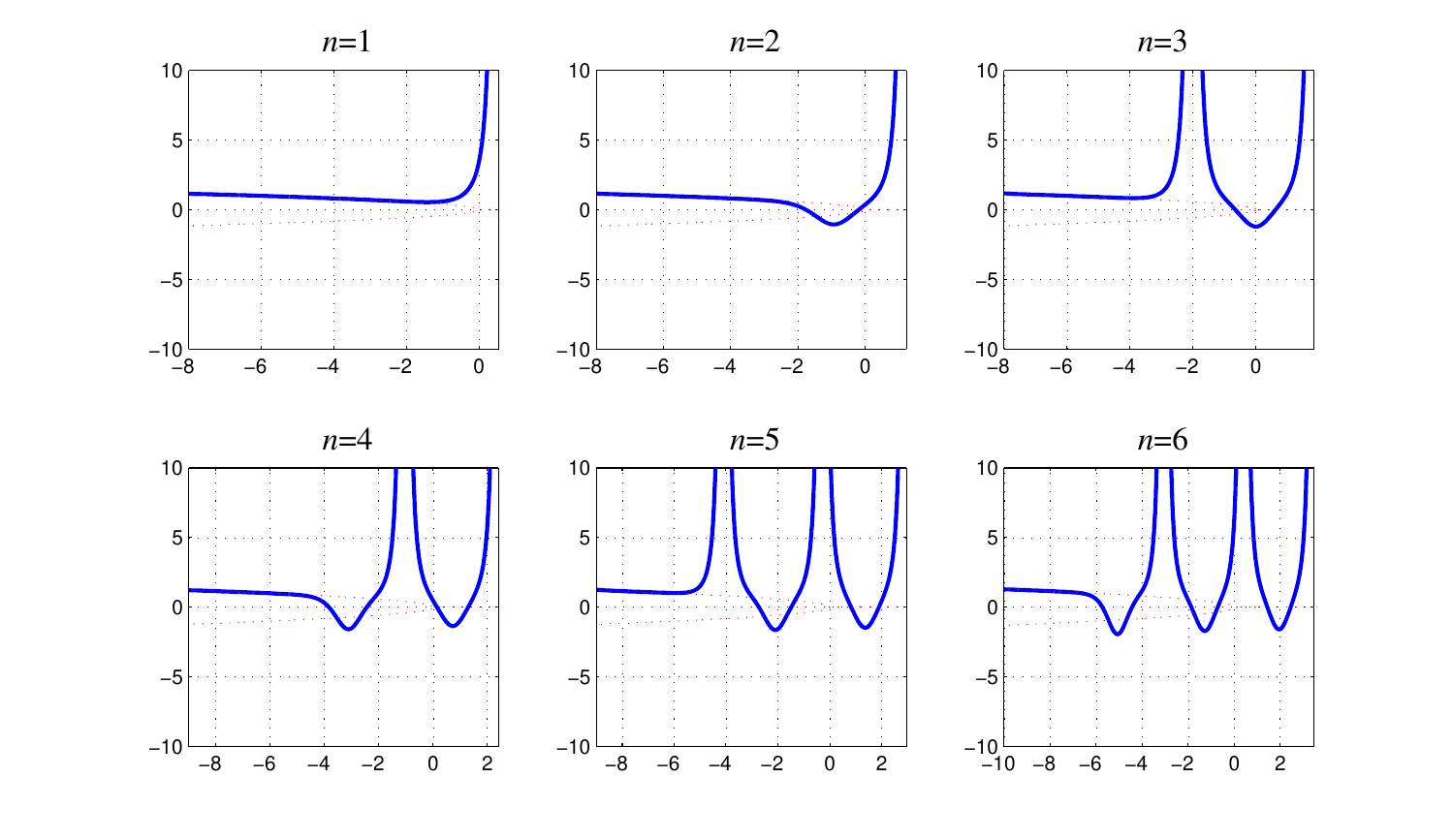}
  \caption{The \PI~solutions $y(t; p, H)$ with $p=2C\cdot\xi_{n}^{4/5}$ and $H=-\frac{\xi_{n}^{6/5}}{7}$, where $C=0.5$.}
  \label{separatrixsolutions}
\end{figure}

\begin{figure}
\centering
  \includegraphics[width=0.8\textwidth]{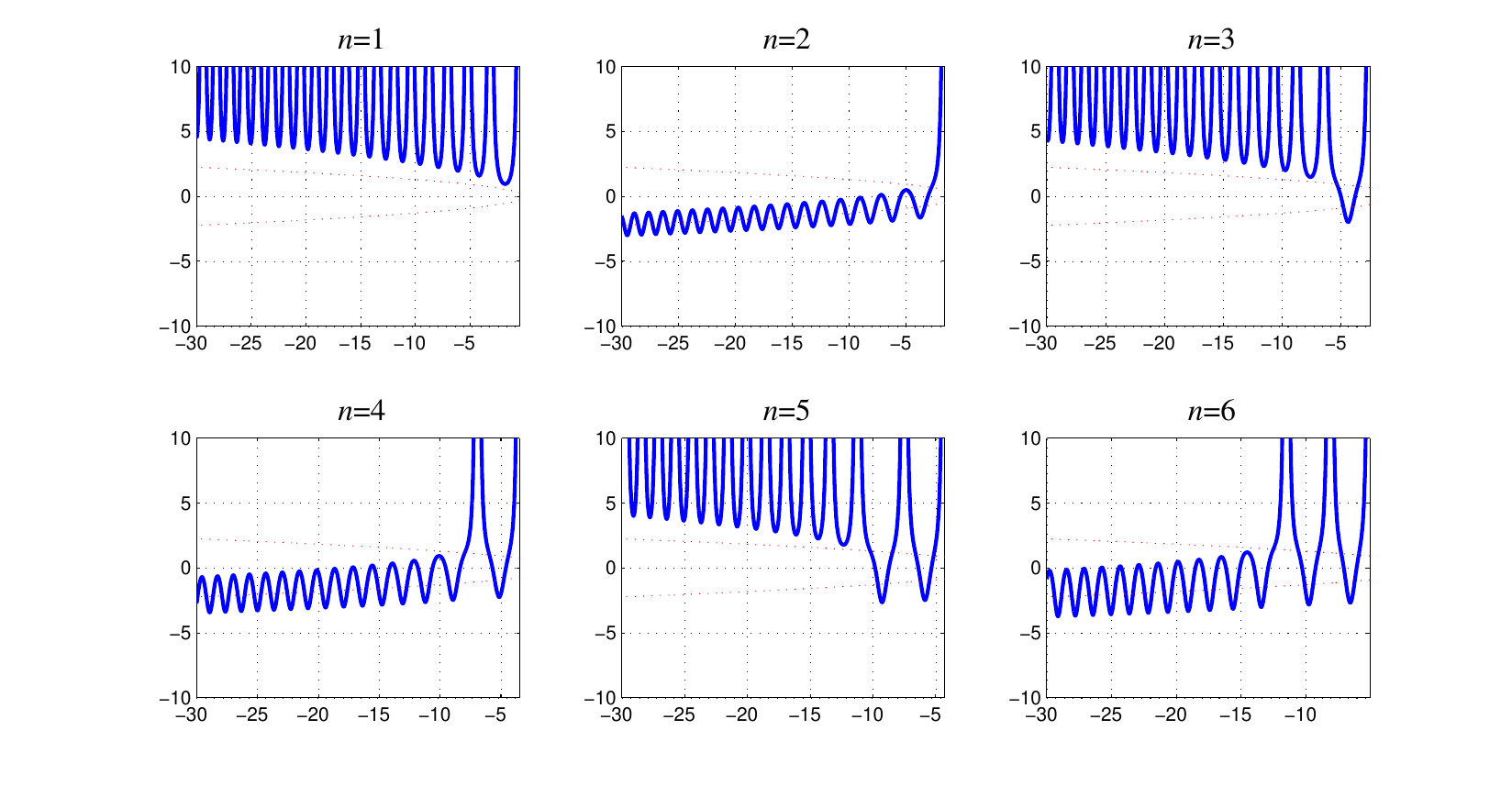}
  \caption{The \PI~solutions $y(t; p, H)$ with $p=2C\left(\tilde{\xi}_{n}\right)^{4/5}$ and $H=-\frac{\left(\tilde{\xi}_{n}\right)^{6/5}}{7}$, where $C=0.5$ and $\tilde{\xi}_{n}=\frac{2n-1-0.5}{\kappa^2}$. Comparing the values of $\tilde{\xi}_{n}$ to $\xi_{n}$ in \cref{eq-def-xi-n}, we see that $\tilde{\xi}_{1}<\xi_{1}<\tilde{\xi}_{2}<\xi_{2}<\cdots <\tilde{\xi}_{5}<\xi_{5}<\tilde{\xi}_{6}<\xi_{6}$.}
  \label{solutions1}
\end{figure}

\begin{proof}
For any fixed $C\in(-3/2^{2/3},C_{0})$, one can choose $C<\delta_{1}<C_{0}$ and $-3/2^{2/3}<\delta_{2}<C$.
{ It then follows from \cref{s-C-region1} or \cref{eq-s0-approx-1}
that there exists $M_{1}>0$ such that $s_{0}:=s_{0}(\xi)$ has a sequence of positive simple zeros on $[M_{1},+\infty)$ if we regard $s_{0}$ as a function of $\xi$. Choose one of the zeros, denoted by $\xi_{1}$, such that $s_{0}'(\xi_{1})>0$. The other zeros larger than $\xi_{1}$ are denoted by $\xi_{n}, n=2,3,\cdots$, arranged in the ascending order. Then we have $s_{0}(\xi)>0$ when $\xi\in\left(\xi_{2n-1},\xi_{2n}\right)$
and $s_{0}<0$ when $\xi\in\left(\xi_{2n},\xi_{2n+1}\right)$.
Moreover, according to \eqref{eq-s0-approx-1}, we know that there exists an integer $N\geq \frac{\kappa^2 M_{1}-1}{2}$ (depending on the selection of $\xi_{1}$) such that
\begin{equation}\label{eq-approx-xi-n}
\xi_{n}\sim\frac{2n+2N-1}{-\kappa^2}\quad \text{as}\quad n\to+\infty.
\end{equation}
Since $N\ll n$, \eqref{eq-def-xi-n} also holds as $n\to+\infty$.}
A combination of these facts and \cref{eq-condition-stokes-multipliers}
yields the first conclusion of \cref{cor-alternate-p-H-large} except for \cref{eq-def-hn}.
Note that $s_{4}=s_{-1}=-\overline{s_{1}}$,
then \cref{eq-def-hn} follows by substituting the asymptotic behavior of $s_{1}$ and $s_{-1}$ into \cref{eq-parameter-h}.

When $C<-3/2^{2/3}$, it is readily seen from \cref{eq-def-kappa-hat-kappa} that $\kappa^2>0$,
and so that $\xi\kappa^2\to+\infty$.
It then follows from \cref{s-C-region1} that $\Im s_{0}=1+s_{2}s_{3}$ remains negative in this case.
When $C>C_{0}$, by the definition of $\hat{\kappa}^2$ in \cref{eq-def-kappa-hat-kappa} or according to \cref{lem-def-C0},
we have $\Re(\hat{\kappa}^2\pi i)>0$ and $\Im(\hat{\kappa}^2\pi i)<0$,
which imply that $\arg{\hat{\kappa}^2}\in\left(\pi,\frac{3}{2}\pi\right)$.
Hence, it follows from \cref{s-C-region2} that
\begin{equation}\label{s2-approx-C>C0}
\begin{aligned}
s_{2}
&\sim -\sqrt{2\pi} i \Gamma\left(\frac{1}{2}+e^{-\pi i}
\frac{\xi\hat{\kappa}^2}{2}\right)\cos\left(\frac{1}{2}\xi\hat{\kappa}^2\right)2^{\frac{\xi\hat{\kappa}^2}{2}}
\xi^{\frac{\xi\hat{\kappa}^2}{2}}e^{-\xi\pi i \hat{\kappa}^2 -2\xi F(\xi)}
\\
&\sim -i\cdot e^{-2\xi I_{F}(\xi)},
\end{aligned}
\end{equation}
and therefore $|s_{2}|>1$ for all $C>C_{0}$ provided that $\xi$ is large enough.
In view of $s_{0}=i(1+s_{2}s_{3})$ and $s_{2}=-\overline{s_{3}}$,
we conclude that, { there exists $M_{2}>0$ such that for all $\xi>M_{2}$}, $\Im s_{0}=1+s_{2}s_{3}$ remains negative in this case.
This completes the proof of \cref{cor-alternate-p-H-large}.
\end{proof}

{ \begin{remark}
In part (i) of \cref{cor-alternate-p-H-large}, we only show the existence of $\xi_{n}$ which corresponding to the separatrix solutions and obtain the large $n$ asymptotic behavior of $\xi_{n}$. Strictly speaking, we do not know how the PI solutions evolve when $\xi$ is finite which is still an open problem. This is because that we only obtain the leading asymptotic behavior of $s_{0}(\xi)$ as $\xi\to+\infty$. Actually, in the proof of  \cref{cor-alternate-p-H-large}, the selection for $\xi_{1}$ is not unique, which also suggests that $\xi_{1}$ may not be the first zero of $s_{0}(\xi)$. Hence, there is an indeterminate integer $N\geq 0$ in \eqref{eq-approx-xi-n} if one intends to use this formula to approximate $\xi_{n}$. Nevertheless, numerical simulation indicates that $\xi_{n}$ can indeed be chosen as the $n$-th zero of $s_{0}(\xi)$.
\end{remark}
}

\begin{figure}[htp]
\centering
  \includegraphics[width=0.8\textwidth]{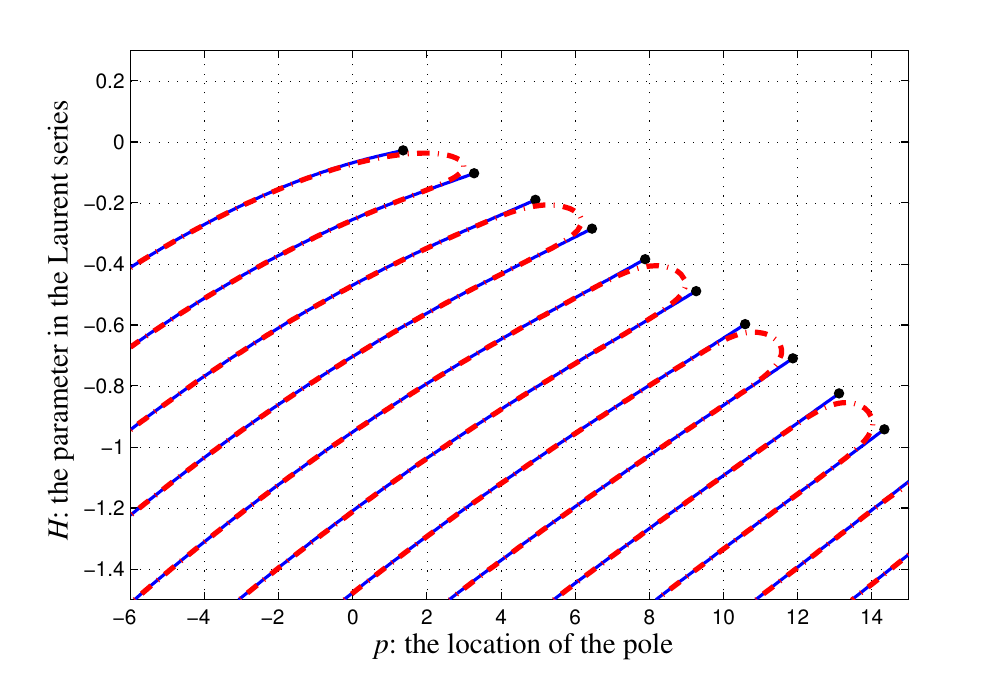}
  \caption{The curves $\Sigma_{n}$ of $(p,H)$ that give rise to separatrix solutions.
  The blue lines are asymptotics of the curves $\Sigma_{n}$ stated in \cref{eq-Sigma-n},
  and the red dash-doted curves are the exact values from the numerical computation.
  The blue lines end at the black points since the parameter $C$ in the equation of $\Sigma_{n}$ is restricted in $(-3/2^{2/3},C_{0})$.}
  \label{Fig-curves-Sigma-n}
\end{figure}

\cref{cor-alternate-p-H-large} gives an asymptotic classification of the \PI~solutions with large negative $H$ and arbitrary $p\in\mathbb{R}$.
Precisely, we see that if $C\in(-3/2^{2/3},C_{0})$,
the three types of solutions appear alternately; see \cref{separatrixsolutions} and \cref{solutions1}.
However, when $C<-3/2^{2/3}$ or $C>C_{0}$, all the solutions are Type (C),
which means that they all have infinite number of poles on the negative real axis.
Note that $\kappa^2$ depends on $C$,
hence $\xi_{n},\, p_{n}$ and $H_{n}$ are also functions of $C$ with $C\in(-3/2^{2/3},C_{0})$.
Therefore, we could restate \cref{cor-alternate-p-H-large} as follows.
There exists a sequence of curves (see \cref{Fig-curves-Sigma-n})
\begin{equation}\label{eq-Sigma-n}
\Sigma_{n}: \begin{cases}p&=p_{n}(C),\\H&=H_{n}(C),\end{cases}\quad C\in(-3/2^{\frac{2}{3}},C_{0})
\end{equation}
on the $(p,H)$ plane such that the \PI~solutions $y(t; p,H)$, with $(p,H)$ lying on these curves, are separatrix solutions.
{ Moreover, we have
\begin{eqnarray}\label{eq-limiting-equation-Sigma-n}
\begin{cases}
p_{n}(C)\sim 2C\left(\frac{2n-\frac{1}{2}}{-\kappa^2}\right)^{\frac{4}{5}},\\
H_{n}(C)\sim -\frac{1}{7}\left(\frac{2n-\frac{1}{2}}{-\kappa^2}\right)^{\frac{6}{5}},\\
\end{cases}\quad C\in(-3/2^{\frac{2}{3}},C_{0})
\end{eqnarray}
as $n\to+\infty$.}
Further more, if the point $(p,H)$ is in the regions between $\Sigma_{2m-1}$ and $\Sigma_{2m}$,
then the corresponding solution $y(t; p,H)$ is Type (A).
Likewise, the solution $y(t; p,H)$ belongs to Type (C) when $(p,H)$ is in the regions between $\Sigma_{2m}$ and $\Sigma_{2m+1}$.
This result of classification is verified by numerical computations; see \cref{Fig-curves-Sigma-n}.
When $C\to -3/2^{2/3}$,
we see that both $p_{n}$ and $H_{n}$ tend to negative infinity;
while as $C\to C_{0}$,
we see $(p_{n},H_{n})$ tends to certain fixed points,
which are denoted by the black points in \cref{Fig-curves-Sigma-n}.

{
\begin{remark}
In the above analysis, we only show the existence of $\Sigma_{n}$ and obtain the limiting form equation \eqref{eq-limiting-equation-Sigma-n} of $\Sigma_{n}$ as $n\to+\infty$ since we only obtain the asymptotic behaivor of $s_{0}(\xi)$ as $\xi\to+\infty$. This is the reason why we say that \cref{cor-alternate-p-H-large} only gives an asymptotic classification of the \PI~solutions. Nevertheless, in \cref{Fig-curves-Sigma-n}, one will find that the numerical and asymptotic curves are very close to each other as $H\to-\infty$ even when $n=1$. It is not a coincidence. In fact, it should be noted that if $\xi\to+\infty$ and $C(\xi)\to -3/2^{2/3}$ simultaneously, we have $\kappa^2\to0$, hence $\frac{-\xi\kappa^2\pi}{2}$ can be asymptotic to any fixed positive constant. It seems that we may choose $\xi_{n}$ with $\xi_{n}= \frac{(2n-1)}{-\kappa^2}\left(1+\mathcal{O}(\xi^{-1})\right)$ for all $n=1,2,\cdots$, not just large $n$.
This should not be confused with \eqref{eq-def-xi-n} where we require $n$ to be large because $\kappa^2$ is fixed there.
\end{remark}
}

The asymptotic behavior of the Stokes multipliers in \cref{Thm-H-negative} is valid not only for fixed $C$,
but also uniformly for all $C(\xi)$ in the corresponding regions,
and thus we may assume $C(\xi)$ depends on $\xi$.
Let $C(\xi)=\frac{1}{2}p\xi^{-4/5}$ with $p$ fixed,
the following corollary is also a direct consequence of \cref{Thm-H-negative},
which gives an asymptotic classification of \PI~solutions for fixed $p$ and large negative $H$.
It can be regarded as another kind of nonlinear eigenvalue phenomenon
similar to the initial value problem in~\cite{Bender-Komijani-2015, Bender-Komijani-Wang-2019, LongLLZ}.
\begin{corollary}\label{cor-classfy-fix-p}
For any fixed $p\in\mathbb{R}$, { there exists $M_{3}>0$ and a sequence $\{\xi_{n}\}$, with $M_{3}<\xi_{1}<\xi_{2}<\cdots$,
such that the \PI~solution $y(t; p, H_{n})$ with $\{H_{n}=-\frac{1}{7}\left(\xi_{n}\right)^{6/5}\}$ belongs to Type (B).}
The asymptotic behaviors of the parameters $h:=h_{n}$ in \cref{eq-behavior-type-B} and $\xi_{n}$
are given in \cref{eq-def-hn} and \cref{eq-def-xi-n} respectively.
In this case, the asymptotic behavior of $\kappa^2$ can be derived more explicitly as
\begin{equation}
\begin{aligned}
\kappa^2:=\kappa^2(\xi_{n})
&=\frac{1}{\pi i}\int_{e^{\frac{\pi i}{3}}}^{e^{-\frac{\pi i}{3}}}4(s^3+1)^{\frac{1}{2}}ds+\mathcal{O}(\xi_{n}^{-4/5})
\\
&=\frac{-4\sqrt{3}}{5\pi}B\left(\frac{1}{2},\frac{1}{3}\right)+\mathcal{O}\left(n^{-\frac{4}{5}}\right)
\end{aligned}
\end{equation}
as $n\to+\infty$.
Furthermore, the \PI~solutions $y(t; p, H)$ with $H_{2m-1}<H<H_{2m}$ are of Type (A)
and the ones with $H_{2m}<H<H_{2m+1}$ are of Type (C).
\end{corollary}

The proof of \cref{cor-classfy-fix-p} is very similar as the one of \cref{cor-alternate-p-H-large}.
The mere difference is that we take $C(\xi)=p \xi^{-1/5}$ with $p$ being fixed in this case.
When $p\to 0$, we can obtain a similar result.
For instance, setting $p=2P\xi^{-1/5}$ with fixed $P\in\mathbb{R}$, {\it i.e.} $C(\xi)=\frac{P}{\xi}$,
we have the following corollary that gives a classification of the \PI~solutions with large $H$ and small $p$.
\begin{corollary}\label{cor-classfy-small-p}
For any fixed $P\in\mathbb{R}$, { there exists $M_{4}>0$ and a sequence $\{\xi_{n}\}$, with $M_{3}<\xi_{1}<\xi_{2}<\cdots$, such that, for each $n$,
the \PI~solution $y(t; p_{n}, H_{n})$ with $p_{n}=2P\xi_{n}^{-1/5}$ and $H_{n}=-\frac{\xi_{n}^{6/5}}{7}$ belongs to Type (B).}
The asymptotic behaviors of the parameters $h:=h_{n}$ in \cref{eq-behavior-type-B} and $\xi_{n}$ are also given in \cref{eq-def-hn} and \cref{eq-def-xi-n} respectively,
and the asymptotic behavior of $\kappa^2$ replaced by
\begin{equation}
\begin{aligned}
\kappa^2:=\kappa^2(\xi_{n})
&=\frac{1}{\pi i}\int_{e^{\frac{\pi i}{3}}}^{e^{-\frac{\pi i}{3}}}4(s^3+1)^{\frac{1}{2}}ds+\mathcal{O}(\xi^{-1})\\
&=\frac{-4\sqrt{3}}{5\pi}B\left(\frac{1}{2},\frac{1}{3}\right)+\mathcal{O}\left(\frac{1}{n}\right)
\end{aligned}
\end{equation}
as $n\to+\infty$.
Furthermore, the \PI~solutions $y(t; p, H)$ with $H_{2m-1}<H<H_{2m}$ and $p_{2m}<p<p_{2m-1}$ are of Type (A)
and the ones with $H_{2m}<H<H_{2m+1}$ and $p_{2m+1}<p<p_{2m}$ are of Type (C).
\end{corollary}

\begin{remark}
In the above three corollaries, we only state the limiting-form connection formulas for $h:=h_{n}$ in the asymptotic behavior of the Type (B) solutions.
One may derive the corresponding connection formulas for the parameters in the asymptotic behaviors of Type (A) and Type (C) solutions,
by a simple substitution of the leading asymptotics for the Stokes multipliers into \cref{eq-parameter-d-theta} and \cref{eq-parameter-rho-sigma}.
\end{remark}

Recall that the Stokes multipliers corresponding to the \PI's real tritronqu\'{e}e solution are $s_{0}=s_{1}=s_{-1}=i$ and $s_{2}=s_{-2}=0$.
From \cref{Thm-H-negative}, we can also obtain the asymptotic behaviors of $p_{n}$ and $H_{n}$ in the Laurent series near the $n$-th pole of this special \PI~solution as $n\to+\infty$.
\begin{corollary}\label{cor-tritronquee}
Let $p_{n}$ be the location of $n$-th pole (in the ascending order) of the real tritronqu\'{e}e \PI~solution
and $H_{n}$ be the free parameter in the corresponding Laurent series.
Then, as $n\to+\infty$, we have
\begin{equation}\label{eq-H-p-Tritronquee}
p_{n}\sim 2C_{0}\left(\frac{4n-2}{\kappa^2(C_{0})}\right)^{4/5}\quad \text{and}\quad H_{n}\sim -\frac{1}{7}\left(\frac{4n-2}{\kappa^2(C_{0})}\right)^{6/5},
\end{equation}
where $C_{0}$ is the constant given in \cref{lem-def-C0}
and $\kappa^2(C_{0})$ is defined in \cref{eq-def-kappa-hat-kappa} with $C(\xi)\equiv C_{0}$.
\end{corollary}

\begin{proof}
According to \cref{s-C-region2}, we see that $s_{2}$ vanishes if and only if $C=C_{0}$.
Moreover, the leading asymptotic behavior of $s_{2}$ can be simplified as
\begin{equation}\label{eq-s2-simplified-C0}
\begin{aligned}
s_{2}
&\sim -\sqrt{2\pi} i \Gamma\left(\frac{1}{2}+e^{-\pi i}\frac{\xi\hat{\kappa}^2}{2}\right)
\cos\left(\frac{1}{2}\xi\hat{\kappa}^2\right)2^{\frac{\xi\hat{\kappa}^2}{2}}\xi^{\frac{\xi\hat{\kappa}^2}{2}}e^{-\xi\pi i \hat{\kappa}^2 -2\xi F(\xi)}
\\
&\sim-i\cos\left(\frac{1}{2}\xi\hat{\kappa}^2\right)e^{-\frac{\xi\pi i \hat{\kappa}^2}{2} -2\xi I_{F}(\xi)}
\end{aligned}
\end{equation}
as $\xi\to+\infty$. Hence, $s_{2}=0$ implies $\frac{\xi\hat{\kappa}^2}{2} \sim \left(n-\frac12\right)\pi$.
In view of the identity $\xi\hat{\kappa}^2=\frac{\xi\kappa^2}{2}$,
we have $\xi=\frac{\left(4n-2\right)\pi}{\kappa^2}$.
Noting that $p=2C\xi^{4/5}$ and $H=-\frac{1}{7}\xi^{6/5}$, we get \cref{eq-H-p-Tritronquee} immediately.
\end{proof}

It is worth mentioning that this solution behaves like $-\sqrt{-t/6}$ as $t\to-\infty$ as a special Type (A) solution,
then a combination of \cref{eq-parameter-d-theta} and \cref{eq-H-p-Tritronquee}
can be regard as the connection formula of the \PI's real tritronqu\'{e}e solution between negative and positive infinity.
\ \\
\paragraph{$\bullet$ Case II: $H\to+\infty$}
\ \\

This case is quite different from the case when $H\to-\infty$.
We find that the leading asymptotic behavior of $s_{0}$ remains purely imaginary and $\Im s_{0}<0$ uniformly for all $p$. The corresponding result is stated as follows.
\begin{theorem}\label{Thm-H-positive}
Suppose that $y(t; p, H)$ is the real solution of \cref{PI equation} with a pole at $t=p$ and $H$ being the free parameter in the Laurent series.
Set $H=\frac{\xi^{6/5}}{7}$ and $p=2C(\xi)\xi^{4/5}$.
Then the asymptotic behavior of the Stokes multiplier $s_0$ corresponding to $y(t; p,H)$ is given by
\begin{equation}
s_{0}=-i\cdot e^{-\xi G(\xi)}\left(1+\mathcal{O}(\xi^{-\frac{1}{2}})\right)
\end{equation}
as $\xi\to+\infty$ uniformly for all $C(\xi)\in\mathbb{R}$, where
\begin{equation}\label{eq-def-G(xi)}
\begin{aligned}
G(\xi)=2\int_{\eta_{0}}^{\infty e^{i\theta}}&\left[\left(s^3+C(\xi)s+1\right)^{\frac{1}{2}}-\left(s^{\frac{3}{2}}+\frac{C(\xi)}{2}s^{-\frac{1}{2}}\right)\right]ds\\
&-\left(\frac{4}{5}(\eta_{0})^{\frac{5}{2}}+2C(\xi)(\eta_{0})^{\frac{1}{2}}\right),
\qquad \theta\in \left(-\frac{2\pi}{5},\frac{2\pi}{5}\right).
\end{aligned}
\end{equation}
\end{theorem}

By the classification criterion in~\cref{eq-condition-stokes-multipliers},
we have the following corollary of \cref{Thm-H-positive}.
\begin{corollary}
There exists an $M>0$ such that for all $p\in\mathbb{R}$ and $H>M$ the \PI~solutions $y(t; p, H)$ belong to Type (C).
\end{corollary}

We can only show the existence of $M$ in the above corollary
as we only derive the leading asymptotic behavior of $s_{0}$ in \cref{Thm-H-positive}.
Nevertheless, the numerical simulations shows that the value of $M$ is { close} to $-0.036\,516\,259$.

\section{Uniform asymptotics and proof of the Theorems}
\label{sec:proof-lemma}

In this section we shall prove \cref{Thm-H-negative} and \cref{Thm-H-positive} by the method of \emph{uniform asymptotics}~\cite{APC}.
The argument consists of two major steps.
The first step is to transform the Lax pair equation \cref{eq-fold-Lax-pair} into a second-order Schr\"{o}dinger equation
and to approximate the solutions of this equation with certain well-known special functions.
Indeed, as stated in the previous section, under proper transformations, we obtain the reduced triconfluent Heun equation \cref{Schrodinger-equation-triconfluent-Heun}.
One can regard \cref{Schrodinger-equation-triconfluent-Heun} as either a scalar or a $1\times 2$ vector-form equation.
When $H$ is large negative or positive, the solutions of this equation can be approximated by certain special functions.
Hence, in the second step, we use the known Stokes phenomena of these special functions to calculate the Stokes multipliers of $Y$,
and then calculate those of $\hat{\Phi}$.
A notable difference between the current work from~\cite{LongLLZ} is
that the asymptotic formulas of the Stokes multipliers here
are valid when one parameter is large and the other is arbitrary,
instead of one large and the other fixed in~\cite{LongLLZ}.

\subsection{Case I: \texorpdfstring{$H\to-\infty$}{H to -Infinity}}
Make the scaling $\lambda=\xi^{2/5}\eta$, $H=-\frac{\xi^{6/5}}{7}$ and $p=2C(\xi) \xi^{4/5}$
as $\xi\rightarrow+\infty$ with $C(\xi)\in(-\infty,\delta]$ for some small $\delta>0$.
Then equation \cref{Schrodinger-equation-triconfluent-Heun} becomes
\begin{equation}\label{eq-shrodinger-scaling}
\begin{split}
\frac{d^{2}Y}{d\eta^{2}}&=\xi^{2}\left[4(\eta^3+C(\xi)\eta+1)\right]Y:=\xi^{2}F(\eta,\xi)Y.
\end{split}
\end{equation}
There are three simple turning points, say $\eta_{j}$, $j=0,1,2$, of the above equation.
They are the zeros of $F(\eta,\xi)$ depending on the value of $C(\xi)$
and may be complex-valued.
Assume that $\eta_{0}$ is on the left half-plane
and $\eta_{1}$, $\eta_{2}$ are on the right half-plane.

As usual, the Stokes curves are those on the $\eta$-plane defined by $\Re\int_{\eta_{j}}^{\eta}F(s,\xi)^{\frac{1}{2}}ds=0,\, j=0,1,2$.
By a careful analysis, the limiting state of the Stokes geometry of the quadratic form $F(\eta,\xi)d\eta^2$ as $\xi\rightarrow+\infty$
depends on the locations of $\eta_{j}$, and thus depends on $C(\xi)$.
A complete classification of the Stokes geometry is given in~\cite[Theorem 7]{Masoero-2010}.
\ \\
\paragraph{$\bullet$ When $C(\xi)\in(-\infty,\delta_{1}]$ with $\delta_{1}\in(0,C_{0})$}
\ \\

There are three cases to be considered.
(a) When $-3/2^{2/3}<C(\xi)\leq \delta_{1}$, $\eta_{1}$ and $\eta_{2}$ are separated and form a conjugate pair on the right half-plane.
In this case we further assume that $\Im \eta_{1}>0>\Im\eta_{2}$.
(b) When $C(\xi)\sim -3/2^{2/3}$, $\eta_{1}$ and $\eta_{2}$ are coalescing to $\eta=2^{-1/3}$.
(c) When $C(\xi)<-3/2^{2/3}$, the two turning points $\eta_{1}$ and $\eta_{2}$ are separated and both real.
Without loss of generality, we assume that $\eta_{1}<\eta_{2}$ in this case.
The Stokes geometry of the quadratic form $F(\eta,\xi)d\eta^2$ as $\xi\rightarrow+\infty$ for the above three cases
are described in \cref{Figure-stokes-geometry}.

\begin{figure}[h]
\centering
\subfigure{
\includegraphics[width=0.3\textwidth]{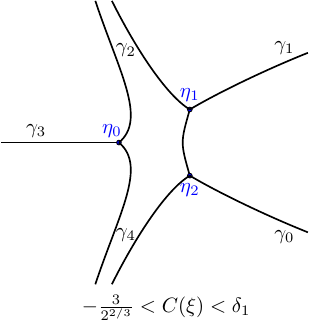}}
\subfigure{
\includegraphics[width=0.3\textwidth]{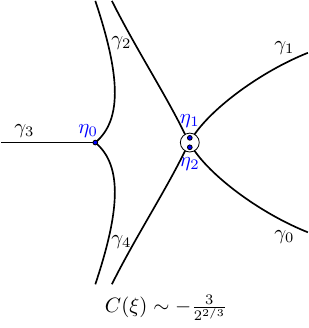}}
\subfigure{
\includegraphics[width=0.3\textwidth]{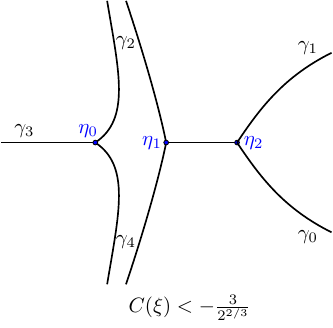}}
\caption{The Stokes geometry of $F(\eta,\xi)d\eta^{2}$ when $C(\xi)\in(-\infty,\delta_{1}]$.}
\label{Figure-stokes-geometry}
\end{figure}

An often used approach, when the two turning points $\eta_{1},\eta_{2}$ of \cref{eq-shrodinger-scaling} are separated,
for deriving asymptotic solutions,
is to use Airy functions to obtain uniform approximations near each turning point,
and then apply a matching technique on the Stokes line joining the two turning points.
When two turning points are coalescing, parabolic cylinder functions are involved \cite{Dunster}.
Hence, one may carry out the analysis case by case for the above three Stokes geometry.
However, it can be done in a unified way regardless of the location of the two turning points.
Although the Stokes geometry changes as $C(\xi)$ varies,
we find that it is possible to approximate the solutions of \cref{eq-shrodinger-scaling},
uniformly for $\eta$ in a larger region that contains both $\eta_{1}$ and $\eta_{2}$.
To this end, we define a mapping $\zeta(\eta)$ by
\begin{equation}\label{zeta-define}
\int_{-\kappa}^{\zeta}(s^2-\kappa^2)^{\frac{1}{2}}ds
=\int_{\eta_{1}}^{\eta}F(s,\xi)^{\frac{1}{2}}ds,
\end{equation}
which is conformal in a neighborhood of $\eta_{1}$.
By the definition of $\kappa$ in \cref{eq-def-kappa-hat-kappa}, we have
\begin{equation}\label{kappa-determin}
\int_{-\kappa}^{\kappa}(s^2-\kappa^2)^{\frac{1}{2}}ds
=\frac{\kappa^2\pi i}{2}=\int_{\eta_{1}}^{\eta_{2}}F(s,\xi)^{\frac{1}{2}}ds.
\end{equation}
Hence, the conformality can be extended to the Stokes curves emanating from $\eta_{1}$ and $\eta_{2}$.
{ Set
\[
p=\frac{d\eta}{d\zeta}=\left(\frac{\zeta^2-\kappa^2}{F(\eta,\xi)}\right)^{\frac{1}{2}},
\quad \varphi(\zeta)=p^{-\frac{1}{2}}Y.
\]
Then we have
\begin{equation}\label{eq-varphi-equation}
\frac{d^2\varphi}{d\zeta^2}=\xi^2(\zeta^2-\kappa^2)\varphi-\frac{1}{2}\left[\frac{p''}{p}-\frac{3}{2}\frac{(p')^2}{p^2}\right]\varphi.
\end{equation}
From \cite[Eqs.~(12.10.35) and (12.10.37)]{NIST-handbook} and \cite[Eqs.~(8.11) and (9.1)]{Olver-1959},
we find that the parabolic cylinder functions $U(\nu,\sqrt{2\xi}\zeta)$ and $V(\nu,\sqrt{2\xi}\zeta)$ are not bounded by a constant as $\xi\to+\infty$,
but instead bounded by two functions of $\xi$ respectively.
Indeed, we have
\begin{equation}
\left|U(\nu,\sqrt{2\xi}\zeta)\right|\leq g(\xi), \quad \left|V(\nu,\sqrt{2\xi}\zeta)\right|\leq h(\xi)
\end{equation}
for all $\eta$ lies on the Stokes lines, where $h(\xi)$ and $g(\xi)$ possess the following asymptotics as $\xi\to+\infty$
\begin{equation}
g(\xi)^2\sim\left\{\begin{aligned}
g_{1}\xi^{-\frac{1}{2}}\Gamma(\frac{1}{2}-\nu) ,\quad&\nu<0,\\
\frac{g_{2}\xi^{-\frac{1}{2}}}{\Gamma(\frac{1}{2}+\nu)},\quad\qquad&\nu>0,
\end{aligned}\right.
\qquad
h(\xi)^2\sim\left\{\begin{aligned}
 \frac{h_{1}\xi^{-\frac{1}{2}}}{\Gamma(\frac{1}{2}-\nu)} ,\quad\qquad &\nu<0,\\
h_{2}\xi^{-\frac{1}{2}}\Gamma(\frac{1}{2}+\nu) ,~&\nu>0
\end{aligned}\right.\end{equation}
with $g_{1}, g_{2}, h_{1}, h_{2}$ being positive constants.
Here, $\nu=\frac{-\xi\kappa^2}{2}$ is the order of the parabolic cylinder functions.
Set
\begin{equation}\label{eq-def-varphi1-varphi2}
\varphi_{1}(\zeta)=\frac{1}{g(\xi)}U(\nu,\sqrt{2\xi}\zeta) \quad \text{and}\quad \varphi_{2}(\zeta)=\frac{1}{h(\xi)}V(\nu,\sqrt{2\xi}\zeta)
\end{equation}
as two linearly independent solutions of $\frac{d^2\varphi}{d\zeta^2}=\xi^2(\zeta^2-\kappa^2)\varphi$.
They are normalized in the sense that they are of $\mathcal{O}(1)$ as $\xi\to+\infty$ uniformly for all $\eta$ on the Stokes lines emanating from $\eta_{1}$ and $\eta_{2}$, {\it i.e.} for all $\zeta$ on the Stokes lines of the parabolic cylinder functions.
Using a similar argument as that of \cite[Theorem 1]{APC}, we have the following lemma.

\begin{lemma}\label{lem-uniform-case-I}
Let $Y$ be any solution of \cref{eq-shrodinger-scaling}, and $\varphi_{1}(\zeta),\varphi_{2}(\zeta)$ be defined as in \eqref{eq-def-varphi1-varphi2}.
Then there are two constants $C_{1}$ and $C_{2}$ such that
\begin{equation}\label{eq-Y-U-V-case-I}
Y=\left(\frac{\zeta^2-\kappa^2}{F(\eta,\xi)}\right)^{\frac{1}{4}}
\left\{\left[C_{1}+r_{1}(\eta,\xi)\right]\varphi_{1}(\zeta)
+\left[C_{2}+r_{2}(\eta,\xi)\right]\varphi_{2}(\zeta)\right\},
\end{equation}
where
\begin{equation}\label{approx-bound-r1}
r_{1}(\eta,\xi), r_{2}(\eta,\xi)=\mathcal{O}\left(\frac{|C_{1}|+|C_{2}|}{\xi}\right)
\end{equation}
as $\xi\to+\infty$, uniformly for $\eta$ on any Stokes lines emanating from $\eta_{1}$ and $\eta_{2}$.
\end{lemma}

\begin{proof}
The proof of \cref{lem-uniform-case-I} is similar to the argument in \cite[pp. 253-255]{APC}. From \eqref{eq-varphi-equation}, we have
\begin{equation}\label{eq-integration-equation-phi}
\begin{split}
\varphi(\zeta)=&C_{1}\varphi_{1}(\zeta)+C_{2}\varphi_{2}(\zeta)\\
&+\int_{-\kappa}^{\zeta}\frac{\varphi_{1}(\zeta)\varphi_{2}(s)-\varphi_{2}(\zeta)\varphi_{1}(s)}{W(s)}\left[\frac{p''}{p}-\frac{3}{2}\frac{(p')^2}{p^2}\right]\varphi(s)ds,
\end{split}
\end{equation}
where
$$\frac{1}{W(s)}=\frac{1}{W(\varphi_{1}(s),\varphi_{2}(s))}=\mathcal{O}(\xi^{-1}).$$
Furthermore, similar as in the proof of \cite[Theorem 1]{APC}, we have
$$\left[\frac{p''}{p}-\frac{3}{2}\frac{(p')^2}{p^2}\right]=\mathcal{O}(\zeta^{-2}),\qquad \text{as}~~\zeta\to\infty.$$ Hence, the last term of \eqref{eq-integration-equation-phi} is integrable. Applying the iteration method used in \cite{APC}, we see that $\varphi(\zeta)$ is bounded.
This fact, together with \cref{eq-integration-equation-phi} and $Y=p^{\frac{1}{2}}\varphi(\zeta)$, leads to the desired result in \cref{lem-uniform-case-I}.
\end{proof}
}

{ \begin{remark}\label{remark-r1-r2}
In the above argument, we find that the last term of \eqref{eq-integration-equation-phi} is integrable when $\eta\to\infty$ with $\arg\eta\sim\frac{(2k-1)\pi}{5}, k=-1,0,1,2$. Then, for any fixed $\xi$, the limits
$\lim\limits_{\eta\to\infty}r_{1}(\eta,\xi) $
and
$\lim\limits_{\eta\to\infty}r_{2}(\eta,\xi)$
both exist.
Moreover, the limit values $\lim\limits_{\eta\to\infty}r_{1}(\eta,\xi)$ and
$\lim\limits_{\eta\to\infty}r_{2}(\eta,\xi)$ with $\arg\eta\sim\frac{(2k-1)\pi}{5}, k=-1,0,1,2$
are different for different $k$. For instance, we know that
\[
r_{1}(\infty e^{\frac{\pi i}{5}},\xi)-r_{1}(\infty e^{-\frac{\pi i}{5}},\xi)
=\mathcal{O}\left(\xi^{-1}\right)
\]
as $\xi\to+\infty$ when $\nu\to-\infty$,
which may not be identically zero.
Here and in what follows, we set $r_{j,k}(\xi)=\lim\limits_{\eta\to\infty}r_{j}(\eta,\xi), j=1,2$, with $\arg\eta\sim \frac{(2k-1)\pi}{5}$, $k=-1,0,1,2$.
\end{remark}
}

The following lemma gives the asymptotics of $\zeta(\eta)$ as $\xi,|\eta|\to+\infty$,
which plays a crucial role in calculating the Stokes multipliers.
The proof of this lemma is left in Appendix B.
\begin{lemma}\label{lemma-zeta-eta-infty-relation}
As $|\eta|\rightarrow+\infty$ with $|\eta|\gg\xi^3$, the asymptotic behavior of $\zeta(\eta)$ is given by
\begin{equation}\label{eq-relation-zeta-eta}
\frac{1}{2}\zeta^2-\frac{\kappa^2}{2}\log{\zeta}=\frac{4}{5}\eta^{\frac{5}{2}}+2C(\xi)\eta^{\frac{1}{2}}+E(\xi)+\mathcal{O}(\eta^{-\frac{1}{2}}),
\end{equation}
where $E(\xi)$ is defined in \cref{Thm-H-negative}.
\end{lemma}

Now we turn to the proof of \cref{Thm-H-negative}.

\textbf{Proof of \cref{Thm-H-negative} (part (i)):}
According to \cite{APC}, in order to calculate $s_{0}$,
we should start with the uniform asymptotics of $Y$ on the two Stokes lines
tending to infinity with $\arg\eta\sim\pm\frac{\pi}{5}$, i.e., $\arg{\zeta}\sim\pm\frac{\pi}{4}$.

When $|\eta|\to+\infty$ with $\arg{\eta}\sim\frac{\pi}{5}$, from \cite[Eqs.~(12.9.1) and (12.9.4)]{NIST-handbook}, we know that when $\arg{z}\sim\frac{\pi}{4}$,
\begin{equation}\label{eq-behavior-parabolic-pi/5}
\left\{\begin{aligned}
U(\nu,z)&\sim e^{-\frac{1}{4}z^2}z^{-\nu-\frac{1}{2}}, \\
V(\nu,z)&\sim \sqrt{\frac{2}{\pi}}e^{\frac{1}{4}z^2}z^{\nu-\frac{1}{2}}+\frac{i}{\Gamma\left(\frac{1}{2}-\nu\right)}e^{-\frac{1}{4}z^2}z^{-\nu-\frac{1}{2}}.
\end{aligned}\right.
\end{equation}
Hence, by substituting \cref{eq-relation-zeta-eta} into \cref{eq-behavior-parabolic-pi/5} and noting that $\lambda=\xi^{\frac{2}{5}}\eta$, we have
\begin{equation}\label{eq-U-V-to-Phi-pi/5}
\left\{\begin{aligned}
\left(\frac{\zeta^2-\kappa^2}{F(\eta,\xi)}\right)^{\frac{1}{4}}{ \varphi_{1}(\zeta)}&\sim c_{1}\frac{-i}{\sqrt{2}}\lambda^{-\frac{3}{4}}e^{-\frac{4}{5}\lambda^{\frac{5}{2}}-p\lambda^{\frac{1}{2}}},\\
\left(\frac{\zeta^2-\kappa^2}{F(\eta,\xi)}\right)^{\frac{1}{4}}{ \varphi_{2}(\zeta)}&\sim c_{2}\frac{-i}{\sqrt{2}}\lambda^{-\frac{3}{4}}e^{-\frac{4}{5}\lambda^{\frac{5}{2}}-p\lambda^{\frac{1}{2}}}+c_{3}\frac{-i}{\sqrt{2}}\lambda^{-\frac{1}{4}}e^{\frac{4}{5}\lambda^{\frac{5}{2}}+p\lambda^{\frac{1}{2}}},
\end{aligned}\right.
\end{equation}
where
{
\begin{equation}
\begin{split}
c_{1}&=\frac{i}{g(\xi)}2^{-\frac{\nu}{2}-\frac{1}{4}}\xi^{\frac{1}{20}-\frac{\nu}{2}}e^{-\xi E(\xi)},\\
c_{2}&=\frac{1}{h(\xi)}\frac{i}{\Gamma\left(\frac{1}{2}-\nu\right)}2^{-\frac{\nu}{2}-\frac{1}{4}}\xi^{\frac{1}{20}-\frac{\nu}{2}}e^{-\xi E(\xi)},\\
c_{3}&=\frac{i}{h(\xi)}\sqrt{\frac{2}{\pi}}2^{\frac{\nu}{2}-\frac{1}{4}}\xi^{\frac{1}{20}+\frac{\nu}{2}}e^{\xi E(\xi)}.
\end{split}
\end{equation}}
It then follows from \cref{lemma-zeta-eta-infty-relation}, \cref{remark-r1-r2} and \cref{eq-canonical-solutions-near-pole} that
\begin{equation}\label{eq-uniform-asymptotic-phik}
(\hat\Phi_{k})_{11}\sim\frac{1}{\sqrt{2}}\lambda^{-\frac{1}{4}}e^{\frac{4}{5}\lambda^{\frac{5}{2}}+p\lambda^{\frac{1}{2}}}\quad\text{and}\quad (\hat\Phi_{k})_{12}\sim\frac{-1}{\sqrt{2}}\lambda^{-\frac{1}{4}}e^{-\frac{4}{5}\lambda^{\frac{5}{2}}-p\lambda^{\frac{1}{2}}},\quad k\in\mathbb{Z},
\end{equation}
as $\lambda\rightarrow\infty$,
and we conclude that
\begin{equation}\label{eq-Y-PHI11-PHI12-pi/5}
\begin{aligned}
Y\sim \left[(C_{1}+r_{1,1}(\xi))c_{1}+(C_{2}+r_{2,1}(\xi))c_{2}\right](\hat{\Phi}_{1})_{12}+(C_{2}+r_{2,1}(\xi))c_{3}(\hat{\Phi}_{1})_{11}.
\end{aligned}
\end{equation}

When $|\eta|\to+\infty$ with $\arg{\eta}\sim-\frac{\pi}{5}$, $\arg{z}\sim-\frac{\pi}{4}$,
the asymptotic behaviors of $U(\nu,z)$ and $V(\nu,z)$ are given by
\begin{equation}\label{eq-behavior-parabolic--pi/5}
\left\{\begin{aligned}
U(\nu,z)&\sim e^{-\frac{1}{4}z^2}z^{-\nu-\frac{1}{2}}, \\
V(\nu,z)&\sim \sqrt{\frac{2}{\pi}}e^{\frac{1}{4}z^2}z^{\nu-\frac{1}{2}}-\frac{i}{\Gamma\left(\frac{1}{2}-\nu\right)}e^{-\frac{1}{4}z^2}z^{-\nu-\frac{1}{2}}.
\end{aligned}\right.
\end{equation}
Hence, we have
\begin{equation}\label{eq-U-V-to-Phi--pi/5}
\left\{\begin{aligned}
\left(\frac{\zeta^2-\kappa^2}{F(\eta,\xi)}\right)^{\frac{1}{4}}{ \varphi_{1}(\zeta)}&\sim c_{1}\frac{-i}{\sqrt{2}}\lambda^{-\frac{3}{4}}e^{-\frac{4}{5}\lambda^{\frac{5}{2}}-p\lambda^{\frac{1}{2}}},\\
\left(\frac{\zeta^2-\kappa^2}{F(\eta,\xi)}\right)^{\frac{1}{4}}{ \varphi_{2}(\zeta)}&\sim -c_{2}\frac{-i}{\sqrt{2}}\lambda^{-\frac{3}{4}}e^{-\frac{4}{5}\lambda^{\frac{5}{2}}-p\lambda^{\frac{1}{2}}}+c_{3}\frac{-i}{\sqrt{2}}\lambda^{-\frac{3}{4}}e^{\frac{4}{5}\lambda^{\frac{5}{2}}+p\lambda^{\frac{1}{2}}}.
\end{aligned}\right.
\end{equation}
Substituting these approximations into \cref{eq-Y-U-V-case-I}
and using \cref{eq-uniform-asymptotic-phik} again, we have
\begin{equation}\label{eq-Y-PHI11-PHI12--pi/5}
Y\sim\left[(C_{1}+r_{1,0}(\xi))c_{1}-(C_{2}+r_{2,0}(\xi))c_{2}\right](\hat{\Phi}_{0})_{12}+(C_{2}+r_{2,0}(\xi))c_{3}(\hat{\Phi}_{0})_{11}.
\end{equation}
A combination of \cref{eq-Y-PHI11-PHI12-pi/5}, \cref{eq-Y-PHI11-PHI12--pi/5} and $\hat\Phi_{1}(\lambda)=\hat\Phi_{0}(\lambda)S_{0}$ yields $r_{2,0}(\xi)=r_{2,1}(\xi)$ and
\begin{equation}
s_{0}=\frac{-(2C_{2}+r_{2,0}(\xi)+r_{2,1}(\xi))c_{2}+(r_{1,0}(\xi)-r_{1,1}(\xi))c_{1}}{(C_{2}+r_{2,1}(\xi))c_{3}}.
\end{equation}
In view of the approximation of $r_{j}(\xi,\eta),j=1,2$, in \cref{approx-bound-r1}, we have
\begin{equation}
\begin{aligned}
s_{0}&=-\frac{2c_{2}}{c_{3}}+\frac{(r_{1,0}(\xi)-r_{1,1}(\xi))}{(C_{2}+r_{2,1}(\xi))}\frac{c_{1}}{c_{3}}\\
&=-i\sqrt{\frac{2}{\pi}}\left[\frac{1}{\Gamma\left(\frac{1}{2}-\nu\right)}+R_{0}(\xi)\right]2^{-\nu}\xi^{-\nu}e^{-2\xi E(\xi)}
\end{aligned}
\end{equation}
as $\xi\to+\infty$, and the estimate of the error bound $R_{0}(\xi)$ is given in~\cref{approx-bound-R0}.

When $\eta\to\infty$ with $\arg{\eta}\sim\frac{3\pi}{5}$,
it follows from \cref{eq-relation-zeta-eta} that $\arg{\zeta}\sim\frac{3\pi}{4}$.
According to \cite[Eq.~(12.9.3)]{NIST-handbook}, we know that
\begin{equation}
U(\nu,z)\sim e^{-\frac{1}{4}z^2}z^{-\nu-\frac{1}{2}}+\frac{i\sqrt{2\pi}}{\Gamma\left(\frac{1}{2}+\nu\right)}e^{-\nu\pi i} e^{\frac{1}{4}z^2}z^{\nu-\frac{1}{2}}, \quad \arg{z}\sim\frac{3\pi}{4}.
\end{equation}
From \cite[Eqs.~(12.2.20) and (12.9.3)]{NIST-handbook}, we also have
\begin{equation}
\begin{aligned}
V(\nu,z)=&\frac{i}{\Gamma\left(\frac{1}{2}-\nu\right)}U(\nu,z)+\sqrt{\frac{2}{\pi}}e^{\pi i(\frac{\nu}{2}-\frac{1}{4})}U(-\nu,-iz)\\
\sim& \frac{i}{\Gamma\left(\frac{1}{2}-\nu\right)}e^{-\frac{1}{4}z^2}z^{-\nu-\frac{1}{2}}+\left[\frac{-\sqrt{2\pi}e^{-\nu\pi i}}{\Gamma\left(\frac{1}{2}-\nu\right)\Gamma\left(\frac{1}{2}+\nu\right)}+\sqrt{\frac{2}{\pi}}\right]e^{\frac{1}{4}z^2}z^{\nu-\frac{1}{2}}\\
\sim& \frac{i}{\Gamma\left(\frac{1}{2}-\nu\right)}e^{-\frac{1}{4}z^2}z^{-\nu-\frac{1}{2}}+i\sqrt{\frac{2}{\pi}}\sin{(\nu\pi)}e^{-\nu\pi i}e^{\frac{1}{4}z^2}z^{\nu-\frac{1}{2}}.
\end{aligned}
\end{equation}
Then, replacing $z$ by $\sqrt{2\xi}\zeta$ in the last two equations and using \cref{eq-relation-zeta-eta} yield
\begin{equation}\label{eq-U-V-to-Phi-3pi/5}
\left\{\begin{aligned}
\left(\frac{\zeta^2-\kappa^2}{F(\eta,\xi)}\right)^{\frac{1}{4}}{ \varphi_{1}(\zeta)}&\sim d_{1}\frac{-i}{\sqrt{2}}\lambda^{-\frac{3}{4}}e^{-\frac{4}{5}\lambda^{\frac{5}{2}}-p\lambda^{\frac{1}{2}}}+d_{2}\frac{-i}{\sqrt{2}}\lambda^{-\frac{3}{4}}e^{\frac{4}{5}\lambda^{\frac{5}{2}}+p\lambda^{\frac{1}{2}}},\\
\left(\frac{\zeta^2-\kappa^2}{F(\eta,\xi)}\right)^{\frac{1}{4}}{ \varphi_{2}(\zeta)}&\sim d_{3}\frac{-i}{\sqrt{2}}\lambda^{-\frac{3}{4}}e^{-\frac{4}{5}\lambda^{\frac{5}{2}}-p\lambda^{\frac{1}{2}}}+d_{4}\frac{-i}{\sqrt{2}}\lambda^{-\frac{3}{4}}e^{\frac{4}{5}\lambda^{\frac{5}{2}}+p\lambda^{\frac{1}{2}}},
\end{aligned}\right.
\end{equation}
where
{ \begin{equation}\label{eq-d1234}
\begin{aligned}
d_{1}&=c_{1},
\quad
d_{2}=i\pi e^{-\nu\pi i}\frac{h(\xi)}{g(\xi)\Gamma\left(\frac{1}{2}+\nu\right)}c_{3},\\
d_{3}&=c_{2},
\quad
d_{4}=i\sin(\nu\pi)e^{-\nu\pi i}c_{3}.
\end{aligned}
\end{equation}}
A combination of \cref{eq-U-V-to-Phi-3pi/5} and \cref{eq-Y-U-V-case-I} leads to
\begin{equation}\label{eq-Y-PHI11-PHI12-3pi/5}
\begin{aligned}
Y\sim&\left[(C_{1}+r_{1,2}(\xi))d_{2}+(C_{2}+r_{2,2}(\xi))d_{4}\right](\hat{\Phi}_{2})_{11}\\
&+\left[(C_{1}+r_{1,2}(\xi))d_{1}+(C_{2}+r_{2,2}(\xi))d_{3}\right](\hat{\Phi}_{2})_{12}.
\end{aligned}
\end{equation}

Comparing \cref{eq-Y-PHI11-PHI12-pi/5} with \cref{eq-Y-PHI11-PHI12-3pi/5} and noting that $\hat\Phi_{2}(\lambda)=\hat\Phi_{1}(\lambda)S_{1}$, we get
$$s_{1}=\frac{(C_{1}+r_{1,2}(\xi))d_{2}+(C_{2}+r_{2,2}(\xi))d_{4}-(C_{2}+r_{2,1}(\xi))c_{3}}{-\left[(C_{1}+r_{1,2}(\xi))d_{1}+(C_{2}+r_{2,2}(\xi))d_{3}\right]}$$
and
\begin{equation}\label{eq-relation-difference-r1-r2}
(C_{1}+r_{1,2}(\xi))d_{1}+(C_{2}+r_{2,2}(\xi))d_{3}=(C_{1}+r_{1,1}(\xi))c_{1}+(C_{2}+r_{2,1}(\xi))c_{2}.
\end{equation}
We note from \cref{eq-d1234} that $d_{1}=c_{1}$, $d_{3}=c_{2}$ and $-\frac{d_{2}}{d_{1}}=\frac{d_{4}-c_{3}}{d_{3}}$, and therefore
\begin{equation}
\begin{aligned}
s_{1}=&-\frac{d_{2}}{d_{1}}+\frac{(r_{2,2}(\xi)-r_{2,1}(\xi))c_{3}}{-\left[(C_{1}+r_{1,2}(\xi))d_{1}+(C_{2}+r_{2,2}(\xi))d_{3}\right]}\\
=&-\left[\frac{\sqrt{2\pi} i}{\Gamma\left(\frac{1}{2}+\nu\right)}+R_{1}(\xi)\right]2^{\nu}\xi^{\nu}e^{2\xi E(\xi)-\nu\pi i}
\end{aligned}
\end{equation}
as $\xi\to+\infty$, where the estimate of $R_{1}(\xi)$ is given in \cref{approx-bound-R0}.
\ \\
\paragraph{$\bullet$ $C(\xi)\in[\delta_{2},+\infty)$ with $\delta_{2}\in(-3/2^{2/3},0)$}
\ \\

The Stokes geometry of $F(\eta,\xi)d\eta^2$ again has three states
corresponding to $\delta_{2}\leq C(\xi)<C_{0}$, $C(\xi)=C_{0}$ and $C(\xi)>C_{0}$ respectively;
see \cref{Figure-stokes-geometry-II}.
When $0<C(\xi)<C_{0}$, there is a Stokes line connecting $\eta_{1}$ and $\eta_{2}$.
When $C(\xi)=C_{0}$, there are two Stokes lines emanating from $\eta_{0}$ to $\eta_{1}$ and to $\eta_{2}$ respectively
and the corresponding Stokes geometry are called the Boutroux Graph in this case, see~\cite[Figure 1]{Masoero-2010}.
When $C(\xi)>C_{0}$, all the Stokes lines tend to infinity.
\begin{figure}[h]
\centering
\subfigure{
\includegraphics[width=0.3\textwidth]{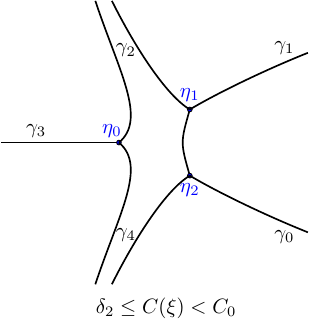}}
\subfigure{
\includegraphics[width=0.3\textwidth]{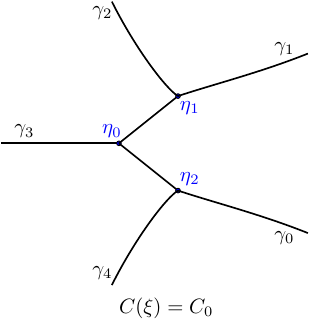}}
\subfigure{
\includegraphics[width=0.3\textwidth]{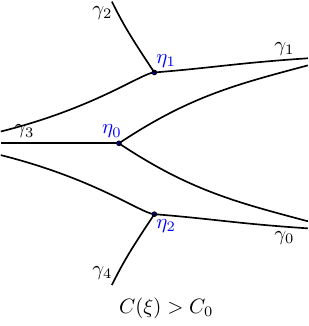}}
\caption{The Stokes geometry of $F(\eta,\xi)d\eta^{2}$ when $C(\xi)\in[\delta_{2},+\infty)$.}
\label{Figure-stokes-geometry-II}
\end{figure}

From the above three figures,
to derive $s_{0}$ in a unified form including all the three cases,
we have to approximate the solutions of \cref{eq-shrodinger-scaling} uniformly in the neighborhoods of $\eta_{i}$, $i=0,1,2$,
which is highly nontrivial.
However, if we intend to derive $s_{1}$ and $s_{2}$ instead of $s_0$,
we only need to work at the turning points $\eta_{0}$ and $\eta_{1}$.
Hence, we define a transformation $\hat{\zeta}(\eta)$ by
\begin{equation}\label{zetahat-define}
\int_{-\hat{\kappa}}^{\hat{\zeta}}(s^2-\hat{\kappa}^2)^{\frac{1}{2}}ds=\int_{\eta_{0}}^{\eta}F(s,\xi)^{\frac{1}{2}}ds,
\end{equation}
which is conformal in a neighborhood of $\eta_{0}$.
By the definition of $\hat{\kappa}$ in \cref{eq-def-kappa-hat-kappa}, we have
\begin{equation}\label{kappahat-determin}
\int_{-\hat{\kappa}}^{\hat{\kappa}}(s^2-\hat{\kappa}^2)^{\frac{1}{2}}ds=\frac{\hat{\kappa}^2\pi i}{2}=\int_{\eta_{0}}^{\eta_{1}}F(s,\xi)^{\frac{1}{2}}ds.
\end{equation}
With this formula, the conformality of $\hat{\zeta}(\eta)$ can be extended to the neighborhood of the Stokes curves emanating from $\eta_{0}$ and $\eta_{1}$.
{ Define
\begin{equation}
\hat{\varphi}_{1}(\hat{\zeta})=\frac{1}{\hat{g}(\xi)}U(\hat\nu,\sqrt{2\xi}\hat\zeta) \quad \text{and}\quad \hat\varphi_{2}(\hat\zeta)=\frac{1}{\hat{h}(\xi)}V(\hat\nu,\sqrt{2\xi}\hat\zeta),
\end{equation}
where $\hat{g}(\xi)$ and $\hat{h}(\xi)$ possess the following asymptotics
\begin{equation}
\hat{g}(\xi)^2\sim
\frac{g_{2}\xi^{-\frac{1}{2}}}{\Gamma(\frac{1}{2}+\hat\nu)},
\qquad
\hat{h}(\xi)^2\sim
h_{2}\xi^{-\frac{1}{2}}\Gamma(\frac{1}{2}+\hat\nu),
\end{equation}
as $\xi\to+\infty$, and $\hat{\nu}=-\frac{\xi\hat{\kappa}^2}{2}$ satisfying $\Re\hat{\nu}>0$.
Then analogously to \cref{lem-uniform-case-I},
we also have the following result.
\begin{lemma}\label{lem-uniform-case-II}
There are two constants $\hat{C}_{1}$ and $\hat{C}_{2}$ such that
\begin{equation}\label{uniform-case-II}
Y=\left(\frac{\hat\zeta^2-\hat{\kappa}^2}{F(\eta,\xi)}\right)^{\frac{1}{4}}\left\{\left[\hat{C}_{1}+\hat{r}_{1}(\eta,\xi)\right]\hat{\varphi}_{1}(\hat{\zeta})+\left[\hat{C}_{2}+\hat{r}_{2}(\eta,\xi)\right]\hat\varphi_{2}(\hat\zeta)\right\},
\end{equation}
where $\hat{\nu}=\frac{-\xi\hat{\kappa}^2}{2}$ satisfying $\Re\hat{\nu}>0$, and
\begin{equation}\label{approx-bound-r1-r2}
\hat{r}_{1}(\eta,\xi),
\hat{r}_{2}(\eta,\xi)=\mathcal{O}\left(\frac{|\hat{C_{1}}|+|\hat{C}_{2}|}{\xi}\right)
\end{equation}
as $\xi\to+\infty$, uniformly for $\eta$ on any adjacent Stokes lines emanating from $\eta_{1}$ and $\eta_{2}$.
\end{lemma}
}
\begin{remark}
It is readily seen that $\hat{r}_{1}(\eta,\xi)$ and $\hat{r}_{2}(\eta,\xi)$ have similar properties
as the ones of $r_{1}(\eta,\xi)$ and $r_{2}(\eta,\xi)$ stated in \cref{remark-r1-r2}.
Precisely speaking, their limit values as $\eta\to\infty$ with $\arg\eta\sim \frac{(1+2k)\pi}{5}$, $k=0,1,2$, exist,
and we denote $\hat{r}_{i,k}(\xi)=\lim\limits_{\eta\to\infty}\hat{r}_{i}(\eta,\xi)$ with $\arg\eta\sim \frac{(1+2k)\pi}{5}$, $k=0,1,2$.
\end{remark}

The proof of this lemma is essentially the same as that of \cref{lem-uniform-case-I} and hence omitted here.
As an analogue of \cref{lemma-zeta-eta-infty-relation}, we can also obtain the asymptotics of $\hat{\zeta}(\eta)$.
\begin{lemma}\label{lemma-hat-zeta-eta-infty-relation}
As $|\eta|\rightarrow+\infty$, the asymptotic behavior of $\hat{\zeta}(\eta)$ is given by
\begin{equation}\label{eq-relation-hat-zeta-eta}
\frac{1}{2}\hat{\zeta}^2-\frac{\hat{\kappa}^2}{2}\log{\hat{\zeta}}=\frac{4}{5}\eta^{\frac{5}{2}}+2C(\xi)\eta^{\frac{1}{2}}+F(\xi)+\mathcal{O}(\eta^{-\frac{1}{2}}),
\end{equation}
where $F(\xi)$ is given in \cref{Thm-H-negative}.
\end{lemma}

\textbf{Proof of \cref{Thm-H-negative} (part (ii)):}
By the last two Lemmas and making use of the asymptotics of $U(\hat{\nu},\sqrt{2\xi}\hat{\zeta})$
and $V(\hat{\nu},\sqrt{2\xi}\hat{\zeta})$ as $\hat{\zeta}\to\infty$
with $\arg\hat{\zeta}\sim\frac{\pi}{4}$ and $\frac{3\pi}{4}$, we get
\begin{equation}\label{eq-U-V-to-Phi-pi/5-hat-zeta}
\left\{\begin{aligned}
\left(\frac{\hat\zeta^2-\hat\kappa^2}{F(\eta,\xi)}\right)^{\frac{1}{4}}{ \hat{\varphi}_{1}(\zeta)}&\sim \hat{c}_{1}\frac{-i}{\sqrt{2}}\lambda^{-\frac{3}{4}}e^{-\frac{4}{5}\lambda^{\frac{5}{2}}-p\lambda^{\frac{1}{2}}},\\
\left(\frac{\hat\zeta^2-\hat\kappa^2}{F(\eta,\xi)}\right)^{\frac{1}{4}}{ \hat{\varphi}_{2}(\zeta)}&\sim \hat{c}_{2}\frac{-i}{\sqrt{2}}\lambda^{-\frac{3}{4}}e^{-\frac{4}{5}\lambda^{\frac{5}{2}}-p\lambda^{\frac{1}{2}}}+\hat{c}_{3}\frac{-i}{\sqrt{2}}\lambda^{-\frac{1}{4}}e^{\frac{4}{5}\lambda^{\frac{5}{2}}+p\lambda^{\frac{1}{2}}},
\end{aligned}\right.
\end{equation}
and
\begin{equation}\label{eq-U-V-to-Phi-3pi/5-hat-zeta}
\left\{\begin{aligned}
\left(\frac{\hat\zeta^2-\hat\kappa^2}{F(\eta,\xi)}\right)^{\frac{1}{4}}{ \hat{\varphi}_{1}(\zeta)}&\sim \hat{d}_{1}\frac{-i}{\sqrt{2}}\lambda^{-\frac{3}{4}}e^{-\frac{4}{5}\lambda^{\frac{5}{2}}-p\lambda^{\frac{1}{2}}}+\hat{d}_{2}\frac{-i}{\sqrt{2}}\lambda^{-\frac{3}{4}}e^{\frac{4}{5}\lambda^{\frac{5}{2}}+p\lambda^{\frac{1}{2}}},\\
\left(\frac{\hat\zeta^2-\hat\kappa^2}{F(\eta,\xi)}\right)^{\frac{1}{4}}{ \hat{\varphi}_{2}(\zeta)}&\sim \hat{d}_{3}\frac{-i}{\sqrt{2}}\lambda^{-\frac{3}{4}}e^{-\frac{4}{5}\lambda^{\frac{5}{2}}-p\lambda^{\frac{1}{2}}}+\hat{d}_{4}\frac{-i}{\sqrt{2}}\lambda^{-\frac{3}{4}}e^{\frac{4}{5}\lambda^{\frac{5}{2}}+p\lambda^{\frac{1}{2}}},
\end{aligned}\right.
\end{equation}
{
where
\begin{equation}
\begin{aligned}
\hat{c}_{1}&=\frac{i}{\hat{g}(\xi)}2^{-\frac{\hat{\nu}}{2}-\frac{1}{4}}\xi^{\frac{1}{20}-\frac{\hat{\nu}}{2}}e^{-\xi F(\xi)},\\
\hat{c}_{2}&=\frac{1}{\hat{h}(\xi)}\frac{i}{\Gamma\left(\frac{1}{2}-\hat{\nu}\right)}2^{-\frac{\hat{\nu}}{2}-\frac{1}{4}}\xi^{\frac{1}{20}-\frac{\hat{\nu}}{2}}e^{-\xi E(\xi)}, \\
\hat{c}_{3}&=\frac{i}{\hat{h}(\xi)}\sqrt{\frac{2}{\pi}}2^{\frac{\hat{\nu}}{2}-\frac{1}{4}}\xi^{\frac{1}{20}+\frac{\hat{\nu}}{2}}e^{\xi F(\xi)}
\end{aligned}
\end{equation}
and
\begin{equation*}
\hat{d}_{1}=\hat{c}_{1},\quad
\hat{d}_{2}=i\pi e^{-\hat\nu\pi i}\frac{\hat{h}(\xi)}{\hat{g}(\xi)\Gamma\left(\frac{1}{2}+\nu\right)}\hat{c}_{3},
\quad
\hat{d}_{3}=\hat{c}_{2},
\quad
\hat{d}_{4}=i\sin(\hat{\nu}\pi)e^{-\hat{\nu}\pi i}\hat{c}_{3}.
\end{equation*}
}
Substituting these approximations into \cref{uniform-case-II} and using \cref{eq-uniform-asymptotic-phik}, we get
\begin{equation}\label{eq-Y-PHI11-PHI12-pi/5-hat}
\begin{aligned}
Y\sim \left[(\hat{C}_{1}+\hat{r}_{1,1}(\xi))\hat{c}_{1}+(\hat{C}_{2}+\hat{r}_{2,1}(\xi))\hat{c}_{2}\right](\hat{\Phi}_{1})_{12} +(\hat{C}_{2}+\hat{r}_{2,1}(\xi))\hat{c}_{3}(\hat{\Phi}_{1})_{11},
\end{aligned}
\end{equation}
and
\begin{equation}\label{eq-Y-PHI11-PHI12-3pi/5-hat}
\begin{aligned}
Y\sim\left[(\hat{C}_{1}+\hat{r}_{1,2}(\xi))\hat{d}_{2}\right.&\left.+(\hat{C}_{2}+\hat{r}_{2,2}(\xi))\hat{d}_{4}\right](\hat{\Phi}_{2})_{11}\\
&+\left[(\hat{C}_{1}+\hat{r}_{1,2}(\xi))\hat{d}_{1}+(\hat{C}_{2}+\hat{r}_{2,2}(\xi))\hat{d}_{3}\right](\hat{\Phi}_{2})_{12}.
\end{aligned}
\end{equation}
In a similar manner as that in deriving $s_{1}$ in the case $C(\xi)\in(-\infty,\delta_{1}]$, we have
\begin{equation}\label{s1-case-II}
s_{1}=-\left[\frac{\sqrt{2\pi} i}{\Gamma\left(\frac{1}{2}+\hat{\nu}\right)}+\hat{R}_{1}(\xi)\right]2^{\hat{\nu}}\xi^{\hat{\nu}}e^{2\xi F(\xi)-2\hat{\nu}\pi i}
\end{equation}
as $\xi\to+\infty$, uniformly for all $C(\xi)\in[\delta_{2},+\infty)$,
and the estimate of $\hat{R}_{1}(\xi)$ is given in \cref{Thm-H-negative}.

A combination of \cite[Eqs.~(12.2.15), (12.2.16), (12.9.1), and (12.9.4)]{NIST-handbook} yields
\begin{equation}
\begin{aligned}
U(\hat{\nu},z)&\sim e^{(2\hat{\nu}+1)\pi i}e^{-\frac{1}{4}z^2}z^{-\hat{\nu}-\frac{1}{2}}+\frac{\sqrt{2\pi}e^{\left(\frac{1}{2}-\hat{\nu}\right)\pi i}}{\Gamma\left(\frac{1}{2}+\hat{\nu}\right)}e^{\frac{1}{4}z^2}z^{\hat{\nu}-\frac{1}{2}},\\
V(\hat{\nu},z)&\sim \frac{ie^{2\hat{\nu}\pi i }}{\Gamma\left(\frac{1}{2}-\hat{\nu}\right)}e^{-\frac{1}{4}z^2}z^{-\hat{\nu}-\frac{1}{2}}+\sin(\hat{\nu}\pi)\sqrt{\frac{2}{\pi}}e^{\pi i(\frac{1}{2}-\hat{\nu})}e^{\frac{1}{4}z^2}z^{\hat{\nu}-\frac{1}{2}},
\end{aligned}
\end{equation}
as $z\to\infty$ with $\arg{z}\sim\frac{5\pi}{4}$.
Hence, when $\eta\to\infty$ with $\arg{\eta}\sim\pi$, we get
\begin{equation}\label{eq-U-V-to-Phi-pi-hat-zeta}
\left\{\begin{aligned}
\left(\frac{\hat\zeta^2-\hat\kappa^2}{F(\eta,\xi)}\right)^{\frac{1}{4}}{ \hat{\varphi}_{1}(\zeta)}&\sim \hat{e}_{1}\frac{-i}{\sqrt{2}}\lambda^{-\frac{3}{4}}e^{-\frac{4}{5}\lambda^{\frac{5}{2}}-p\lambda^{\frac{1}{2}}}+\hat{e}_{2}\frac{-i}{\sqrt{2}}\lambda^{-\frac{3}{4}}e^{\frac{4}{5}\lambda^{\frac{5}{2}}+p\lambda^{\frac{1}{2}}},\\
\left(\frac{\hat\zeta^2-\hat\kappa^2}{F(\eta,\xi)}\right)^{\frac{1}{4}}{ \hat{\varphi}_{2}(\zeta)}&\sim \hat{e}_{3}\frac{-i}{\sqrt{2}}\lambda^{-\frac{3}{4}}e^{-\frac{4}{5}\lambda^{\frac{5}{2}}-p\lambda^{\frac{1}{2}}}+\hat{e}_{4}\frac{-i}{\sqrt{2}}\lambda^{-\frac{3}{4}}e^{\frac{4}{5}\lambda^{\frac{5}{2}}+p\lambda^{\frac{1}{2}}},
\end{aligned}\right.
\end{equation}
where
\begin{equation}\label{eq-relation-e-d}
\begin{aligned}
\hat{e}_{1}=\hat{d}_{1}e^{(2\hat{\nu}+1)\pi i},\quad \hat{e}_{2}=\hat{d}_{2},\quad \hat{e}_{3}=\hat{d}_{3}e^{2\hat{\nu}\pi i},\quad \hat{e}_{4}=\hat{d}_{4}.
\end{aligned}
\end{equation}
Substituting \cref{eq-U-V-to-Phi-pi-hat-zeta} into \cref{uniform-case-II} and noting \cref{eq-uniform-asymptotic-phik}, we further obtain
\begin{equation}\label{eq-Y-PHI11-PHI12-pi-hat}
\begin{aligned}
Y\sim\left[(\hat{C}_{1}+\hat{r}_{1,3}(\xi))\hat{e}_{2}\right.&\left.+(\hat{C}_{2}+\hat{r}_{2,3}(\xi))\hat{e}_{4}\right](\hat{\Phi}_{3})_{11}\\
&+\left[(\hat{C}_{1}+\hat{r}_{1,3}(\xi))\hat{e}_{1}+(\hat{C}_{2}+\hat{r}_{2,3}(\xi))\hat{e}_{3}\right](\hat{\Phi}_{3})_{12}.
\end{aligned}
\end{equation}
Combining \cref{eq-Y-PHI11-PHI12-3pi/5-hat} with \cref{eq-Y-PHI11-PHI12-pi-hat}
and observing that
\[
\hat{\Phi}_{3}=\hat{\Phi}_{2}\left[\begin{matrix}1&0\\s_{2}&1\end{matrix}\right],
\]
we get the following two equations
\begin{equation}\label{eq-relation-difference-r1-r2-2-3}
(\hat{C}_{1}+\hat{r}_{1,2}(\xi))\hat{d}_{2}+(\hat{C}_{2}+\hat{r}_{2,2}(\xi))d_{4}=(\hat{C}_{1}+\hat{r}_{1,3}(\xi))\hat{e}_{2}+(\hat{C}_{2}+\hat{r}_{2,3}(\xi))\hat{e}_{4}
\end{equation}
and
\begin{equation}\label{eq-s2-representation}
s_{2}
=\frac{(\hat{C}_{1}+\hat{r}_{1,2}(\xi))\hat{d}_{1}+(\hat{C}_{2}+\hat{r}_{2,2}(\xi))\hat{d}_{3}-(\hat{C}_{1}+\hat{r}_{1,3}(\xi))\hat{e}_{1}+(\hat{C}_{2}+\hat{r}_{2,3}(\xi))\hat{e}_{3}}
{(\hat{C}_{1}+\hat{r}_{1,3}(\xi))\hat{e}_{2}+(\hat{C}_{2}+\hat{r}_{2,3}(\xi))\hat{e}_{4}}.
\end{equation}
From \cref{eq-relation-e-d} and \cref{eq-relation-difference-r1-r2-2-3}, it is readily seen that $(\hat{r}_{2,3}(\xi)-\hat{r}_{2,2}(\xi))=\frac{\hat{d}_{2}}{\hat{d}_{4}}(\hat{r}_{1,2}(\xi)-\hat{r}_{1,3}(\xi))$,
which implies $\hat{r}_{1,2}(\xi)-\hat{r}_{1,3}(\xi)=\mathcal{O}(\xi^{-1}\sin(\hat\nu\pi))$ as $\xi\to+\infty$ when $\hat{\nu}>0$.
Making use of this fact and noting that $\frac{\hat{d}_{1}-\hat{e}_{1}}{\hat{e}_{2}}=\frac{\hat{d}_{3}-\hat{e}_{3}}{\hat{e}_{4}}$ from \cref{eq-relation-e-d}, we obtain
\begin{equation}\label{s2-case-II}
\begin{aligned}
s_{2}&=\frac{\hat{d}_{1}-\hat{e}_{1}}{\hat{e}_{2}}+\frac{(\hat{r}_{1,2}(\xi)-\hat{r}_{1,3}(\xi))\hat{d}_{1}-(\hat{r}_{2,3}(\xi)-\hat{r}_{2,2}(\xi))\hat{d}_{3}}{(\hat{C}_{1}+\hat{r}_{1,3}(\xi))\hat{e}_{2}+(\hat{C}_{2}+\hat{r}_{2,3}(\xi))\hat{e}_{4}}\\
&=\left[-\sqrt{\frac{2}{\pi}} i \Gamma\left(\frac{1}{2}+\hat{\nu}\right)\cos\left(\hat{\nu}\pi\right)+\hat{R}_{2}(\xi)\right]2^{-\hat{\nu}}\xi^{-\hat{\nu}}e^{-2\hat{\nu}\pi i-2\xi F(\xi)}
\end{aligned}
\end{equation}
as $\xi\to+\infty$, uniformly for all $C(\xi)\in[\delta_{2},+\infty)$,
and the estimate of $\hat{R}_{2}(\xi)$ is given in \cref{lemma-hat-zeta-eta-infty-relation}.

Finally, when $\delta_{2}\leq C(\xi)\leq \delta_{1}$, the turning points $\eta_{1}$ and $\eta_{2}$ are complex conjugates.
Then from the definitions of $I_{E}(\xi)$ and $I_{F}(\xi)$ in \cref{eq-def-I-E(xi)},
we have $I_{E}(\xi)=\overline{I_{F}(\xi)}$ and
\begin{equation}
\begin{aligned}
I_{F}(\xi)-I_{E}(\xi)=&2\lim\limits_{\eta\to\infty}\int_{\eta_{1}}^{\eta}\left[\left(s^3+C(\xi)s+1\right)^{\frac{1}{2}}
-\left(s^{\frac{3}{2}}+\frac{C(\xi)}{2}s^{-\frac{1}{2}}\right)\right]ds
\\
&-\left(\frac{4}{5}(\eta_{1})^{\frac{5}{2}}+2C(\xi)(\eta_{1})^{\frac{1}{2}}\right)+\left(\frac{4}{5}(\eta_{2})^{\frac{5}{2}}+2C(\xi)(\eta_{2})^{\frac{1}{2}}\right)\\
&-2\lim\limits_{\eta\to\infty}\int_{\eta_{2}}^{\eta}\left[\left(s^3+C(\xi)s+1\right)^{\frac{1}{2}}
-\left(s^{\frac{3}{2}}+\frac{C(\xi)}{2}s^{-\frac{1}{2}}\right)\right]ds\\
=&2\int_{\eta_{1}}^{\eta_{2}}\left(s^3+C(\xi)s+1\right)^{\frac{1}{2}}ds\\
=&\frac{\kappa^2\pi i}{2}.
\end{aligned}
\end{equation}
Hence, $\kappa^2\pi=2\Im(I_{F}(\xi)-I_{E}(\xi))$. To verify $\Re{(\hat{\kappa}^2\pi i+2 I_{F}(\xi))}=0$,
we note that
\begin{equation}\label{eq-hat-kappa-I_F}
\begin{aligned}
\frac{\hat{\kappa}^2\pi i}{2}+ I_{F}(\xi)= &2\int_{\eta_{0}}^{\infty}\left[\left(s^3+C(\xi)s+1\right)^{\frac{1}{2}}-\left(s^{\frac{3}{2}}+\frac{C(\xi)}{2}s^{-\frac{1}{2}}\right)\right]ds\\
&-\left(\frac{4}{5}(\eta_{0})^{\frac{5}{2}}+2C(\xi)(\eta_{0})^{\frac{1}{2}}\right)
\end{aligned}
\end{equation}
and the integral path is from $\eta_{0}$ to $\infty$ along the upper edge of the negative real axis.
Hence the integral in \cref{eq-hat-kappa-I_F} is purely imaginary.
Combining this with the fact that $\eta_{0}<0$, we conclude that $\Re{(\hat{\kappa}^2\pi i+2 I_{F}(\xi))}=0$.

\subsection{Case II: \texorpdfstring{$H\to+\infty$}{H to Infinity}}

With the scaling $\lambda=\xi^{2/5}\eta$, $H=\frac{\xi^{6/5}}{7}$ and $p=2C(\xi)\xi^{4/5}$ as $\xi\rightarrow+\infty$,
equation \cref{Schrodinger-equation-triconfluent-Heun} is reduced to
\begin{equation}\label{eq-shrodinger-H-positive}
\frac{d^{2}Y}{d\eta^{2}}=\xi^{2}\left[4(\eta^3+C(\xi)\eta-1)\right]Y:=\xi^{2}\tilde{F}(\eta,\xi)Y.
\end{equation}
There are three turning points, say $\eta_{j}$, $j=0,1,2$,
where $\eta_{0}$ is in the right half-plane and $\eta_{1},\eta_{2}$ are in the left half-plane.

According to \cite{Masoero-2010}, there are also three limiting states of the Stokes geometry of the quadratic form $F(\eta,\xi)d\eta^2$ as $\xi\rightarrow+\infty$,
which are described in \cref{Figure-stokes-geometry-III}.
When $C(\xi)>-3/2^{\frac{2}{3}}$, all the Stokes lines tend to infinity.
When $C(\xi)=-3/2^{2/3}$, the two turning points $\eta_{1}$ and $\eta_{2}$ coalesce to a double turning point.
When $C(\xi)<-3/2^{2/3}$, there is a Stokes line connecting $\eta_{1}$ and $\eta_{2}$.

\begin{figure}[h]
\centering
\subfigure{
\includegraphics[width=0.3\textwidth]{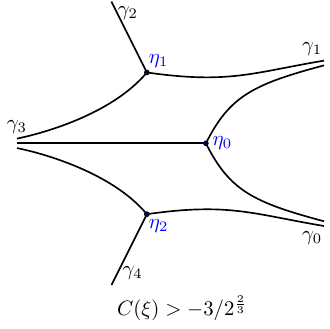}}
\subfigure{
\includegraphics[width=0.3\textwidth]{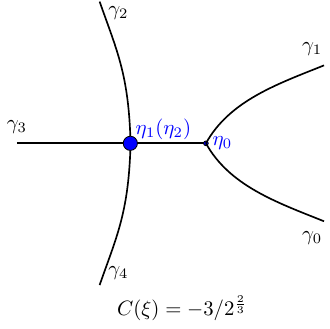}}
\subfigure{
\includegraphics[width=0.3\textwidth]{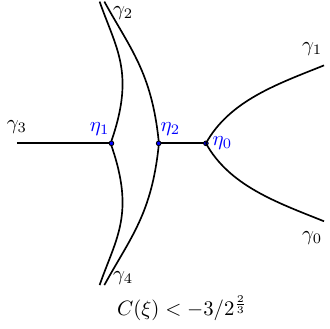}}
\caption{The Stokes geometry of $\tilde{F}(\eta,\xi)d\eta^{2}$.}
\label{Figure-stokes-geometry-III}
\end{figure}

From \cref{Figure-stokes-geometry-III},
we find that in order to calculate $s_{0}$,
we only need to obtain the uniform asymptotics of $Y$ on the two adjacent Stokes emanating from $\eta_{0}$
and tending to infinity with $\arg\eta\sim\pm\frac{\pi}{5}$.
To do so, we define two conformal mappings $\omega(\eta)$ by
\begin{equation}\label{omega-define}
\int_{0}^{\omega}s^{\frac{1}{2}}ds=\int_{\eta_{0}}^{\eta}\tilde{F}(s,\xi)^{\frac{1}{2}}ds,
\end{equation}
along the two Stokes curves.
Then we have the following lemma; see~\cite{APC} or \cite{LongLLZ}.
\begin{lemma}\label{lem-uniform-airy}
There are two constants $D_{1}$ and $D_{2}$ such that
\begin{equation}
Y=\left(\frac{\omega}{F(\eta,\xi)}\right)^{\frac{1}{4}}\left\{D_{1}\left[1+r_{3}(\eta,\xi)\right]\Ai(\xi^{\frac{2}{3}}\omega)+D_{2}\left[1+r_{4}(\eta,\xi)\right]\Bi(\xi^{\frac{2}{3}}\omega)\right\}
\end{equation}
with $r_{3}(\eta,\xi),r_{4}(\eta,\xi)=\mathcal{O}(\xi^{-1/2})$ as $\xi\to+\infty$, uniformly for $\eta$ on any two adjacent Stokes lines emanating from $\eta_{0}$ tending to infinity with $\arg\eta\sim\pm\frac{\pi}{5}$.
\end{lemma}

Moreover, the asymptotics of $\omega(\eta)$ as $\eta\to\infty$ can be derived and stated as follows.
\begin{lemma}\label{lemma-omega-eta-infty-relation}
As $|\eta|\rightarrow\infty$ with $\arg\eta\in \left(-\frac{3\pi}{5},\frac{3\pi}{5}\right)$, we have
\begin{equation}\label{omega-eta-infty-relation}
\frac{2}{3}\omega^{\frac{3}{2}}=\frac{4}{5}\eta^{\frac{5}{2}}+2C(\xi)\eta^{\frac{1}{2}}+G(\xi)+\mathcal{O}\left(\eta^{-\frac{1}{2}}\right),
\end{equation}
where $G(\xi)$ is defined in \cref{Thm-H-positive}.
\end{lemma}

\textbf{Proof of \cref{Thm-H-positive}:}
If $\eta\rightarrow\infty$ with $\arg\eta\sim\frac{\pi}{5}$,
then $\arg\omega\sim\frac{\pi}{3}$.
Hence, by the asymptotics of $\Ai(z)$ and $\Bi(z)$ in \cite[Eqs.~(45) and (47)]{LongLLZ},
noting that $\lambda=\xi^{2/5}\eta$ and the definition of $\tilde{F}(\eta,\xi)$ in \cref{eq-shrodinger-H-positive},
we get
\begin{equation}\label{eq-Phi-pi/5-airy}
\begin{split}
\left(1+r_{3}(\eta,\xi)\right)\left(\frac{\omega}{\tilde{F}(\eta,\xi)}\right)^{\frac{1}{4}}\Ai(\xi^{\frac{2}{3}}\omega)\sim& g_{1}\frac{-i}{\sqrt{2}}\lambda^{-\frac{3}{4}}e^{-\frac{4}{5}\lambda^{\frac{5}{2}}-p\lambda^{\frac{1}{2}}},\\
\left(1+r_{4}(\eta,\xi)\right)\left(\frac{\omega}{\tilde{F}(\eta,\xi)}\right)^{\frac{1}{4}}\Bi(\xi^{\frac{2}{3}}\omega)\sim& ig_{1}\frac{-i}{\sqrt{2}}\lambda^{-\frac{3}{4}}e^{-\frac{4}{5}\lambda^{\frac{5}{2}}-p\lambda^{\frac{1}{2}}}\\
&+2g_{2}\frac{-i}{\sqrt{2}}\lambda^{-\frac{3}{4}}e^{\frac{4}{5}\lambda^{\frac{5}{2}}+p\lambda^{\frac{1}{2}}}
\end{split}
\end{equation}
as $\lambda\rightarrow\infty$ with $\arg\lambda\sim\frac{\pi}{5}$, where
\begin{equation}\label{eq-g1-g2}
g_{1}=i\frac{1}{2\sqrt{\pi}}\xi^{\frac{3}{10}}e^{-\xi G(\xi)+\mathcal{O}(\xi^{-\frac{1}{2}})}\quad \text{and}\quad g_{2}=\frac{i}{2\sqrt{\pi}}\xi^{\frac{3}{10}}e^{\xi G(\xi)+\mathcal{O}(\xi^{-\frac{1}{2}})}
\end{equation}
as $\xi\rightarrow+\infty$.
When $\eta\rightarrow\infty$ with $\arg\eta\sim-\frac{\pi}{5}$, then $\arg\omega\sim-\frac{\pi}{3}$. In a similar way of deriving \cref{eq-Phi-pi/5-airy}, we get
\begin{equation}\label{eq-Phi--pi/5-airy}
\begin{aligned}
\left(1+r_{3}(\eta,\xi)\right)\left(\frac{\omega}{\tilde{F}(\eta,\xi)}\right)^{\frac{1}{4}}\Ai(\xi^{\frac{2}{3}}\omega)
\sim& g_{1}\frac{-i}{\sqrt{2}}\lambda^{-\frac{3}{4}}e^{-\frac{4}{5}\lambda^{\frac{5}{2}}-p\lambda^{\frac{1}{2}}},
\\
\left(1+r_{4}(\eta,\xi)\right)\left(\frac{\omega}{\tilde{F}(\eta,\xi)}\right)^{\frac{1}{4}}\Bi(\xi^{\frac{2}{3}}\omega)
\sim&-ig_{1}\frac{-i}{\sqrt{2}}\lambda^{-\frac{3}{4}}e^{-\frac{4}{5}\lambda^{\frac{5}{2}}-p\lambda^{\frac{1}{2}}}
\\
&+2g_{2}\frac{-i}{\sqrt{2}}\lambda^{-\frac{3}{4}}e^{\frac{4}{5}\lambda^{\frac{5}{2}}+p\lambda^{\frac{1}{2}}}
\end{aligned}
\end{equation}
as $\xi, \lambda\rightarrow\infty$ with $\arg\lambda\sim-\frac{\pi}{5}$.
Here, $g_{i},i=1,2,3$, are not identically but asymptotically equal to the ones in \cref{eq-g1-g2}.
We use the same notations since we only concern with the asymptotics of the Stokes multiplier $s_{0}$.
A combination of \cref{eq-uniform-asymptotic-phik}, \cref{eq-Phi-pi/5-airy}, \cref{eq-Phi--pi/5-airy} and $\hat{\Phi}_{1}=\hat{\Phi}_{0}S_{0}$ yields
\begin{equation}
S_{0}=\begin{bmatrix}1&0\\s_{0}&1\end{bmatrix}
=\begin{bmatrix}1&0\\-i\frac{g_{1}}{g_{2}}&1\end{bmatrix}.
\end{equation}
Hence, we have $s_{0}=-i\frac{g_{1}}{g_{2}}=-i\cdot e^{-\xi G(\xi)+\mathcal{O}(\xi^{-\frac{1}{2}})}$ as $\xi\to+\infty$.
This completes the proof of \cref{Thm-H-positive}.

\section{Numerical analysis}
\label{sec:numerical}

In this section, we present our numerical findings and compare some of them with asymptotic results stated in the previous sections.
This can be considered as the continuation of several previous numerical studies on the first { Painlev\'e} equation \cite{Bender-Komijani-2015,Fornberg-Weideman}.
Here we focus on the solutions on the real $t$-axis.
Instead of starting from the origin using initial conditions $y(0)$ and $y'(0)$ as in the previous studies,
our computation starts from a point near a movable pole characterized by a pair of real parameters $(p,H)$.
In practice, we pick a point close to a pole, say $t_0=p + \epsilon$ or $t_0=p - \epsilon$,
where $\epsilon$ is small enough so that the point is away from other poles,
and large enough to have a reasonable numerical value of the solution.
In this part of the computation, we choose $\epsilon=0.1$.
We use the Laurent series in \cref{eq-Laurent-series} to evaluate the initial value $y(t_0)$ and the initial slope $y'(t_0)$.
To improve the accuracy, we actually used 40 terms in the Laurent expansion, which is too long to be included in this paper.
We then numerically integrate the first Painlev\'e equation away from the starting pole.
If we come close to another singularity, we circumvent it by going off to the complex $t$-plane.
Since all the singularities are double poles, the solution becomes real again after a half turn ($180^{\circ}$).
These series of exercises are repeated until a solution in a desired region is fully obtained.

Combining with the above computing algorithm and the binary searching,
we find that there exists a sequence of curves on the $(p,H)$ plane that give rise to the separatrix solutions [Type (B) solutions];
see the solid black lines in \cref{oscillate-region-fig}.
When $(p,H)$ is located in the finger-like (blue) regions enclosed by these curves,
the corresponding solutions oscillate when $t\to-\infty$ [Type (A) solutions].
Otherwise, the blank region leads to singular solutions [Type (C) solutions] that have infinity double poles on the negative real axis.
These facts verify \cref{cor-alternate-p-H-large}; see the comparison between these curves and the asymptotic of $\Sigma_{n}, n=1,2,\cdots$ in \cref{Fig-curves-Sigma-n}. One may find that it essentially gives a complete classification of the \PI~solutions in terms of $p$ and $H$.
These regions in the $(p,H)$ plane are analogous to a phase diagram in physics.

\begin{figure}
  \centering
	\includegraphics[width=0.9\textwidth]{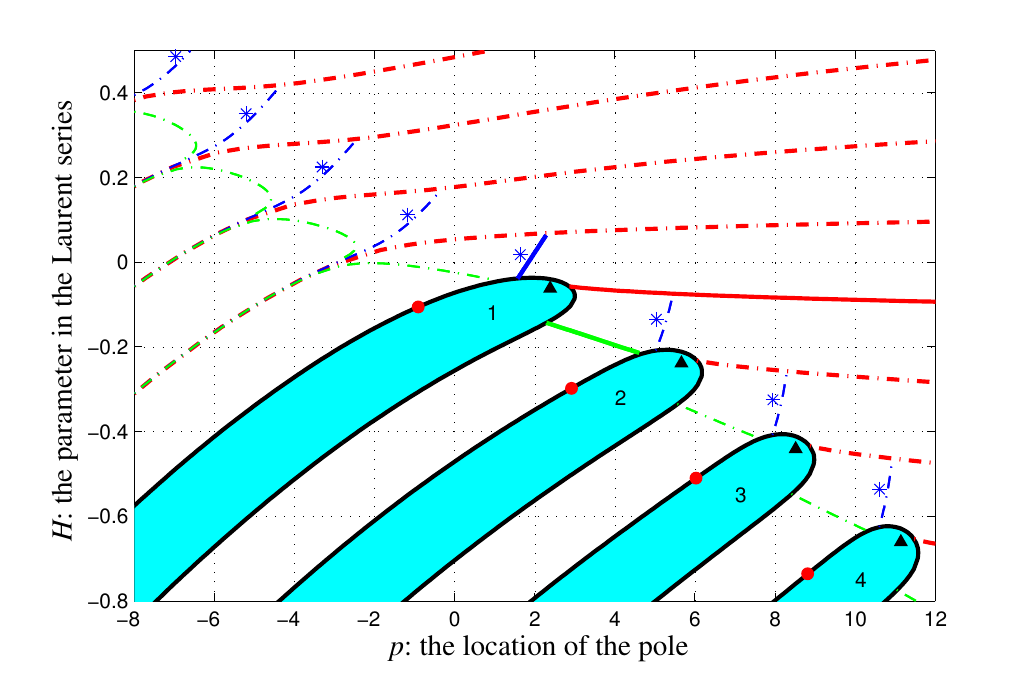}
	\caption{The phase diagram of \PI~solutions in the $(p,H)$ plane.
             The points $(p,H)$ in the blue finger-shaped regions lead to oscillating solutions,
             the points on the boundaries of fingers (solid black lines) lead to separatrix solutions,
             and the points in the blank regions lead to singular solutions.
             The red solid line is an arbitrarily chosen partition line with the minimum value of the solution (on the left of the initial pole) being zero. The red dash-dotted lines are locations of the other poles of the solutions started from the red solid line.
             Points in different strip regions are connected by \PI~solutions. For instance, all the small triangles belong to the tritronqu\'{e}e solution; all the red dots belong to a separatrix solution;
             and all the blue asterisks belong to a singular solution. Moreover, the blue (green) solid line is mapped to the blue (green) dashed-dotted curves by the corresponding singular \PI~solutions. The points $(p,H)$ for the most left pole of oscillating \PI~solutions are located in the finger-shaped region marked $1$, and the corresponding region marked $n$ is filled with points $(p,H)$ for the $n$-th pole (in the ascending order) of the oscillating \PI~solutions.}
   \label{oscillate-region-fig}
\end{figure}

Every \PI~solution has infinite double poles in the positive $t$-axis.
Therefore, a solution corresponds to infinite points in the parameter space $(p,H)$.
For example, all the small triangles in \cref{oscillate-region-fig} correspond to the tritronqu\'{e}e solution,
all the red points correspond to the same separatrix solution,
and all the blue asterisks correspond to the same singular solutions.
In fact, every Type (A) solution corresponds to one and only one point in each finger-shaped region.
The same goes to every Type (B) solution.
In this sense, each finger is ``mapped'' to each other finger by these solutions.
It is natural to divide the whole parameter plane into infinite parts and one \PI~solution corresponds to one and only one point in each part.
Such divisions can be extended to the blank region and are not unique to some extent.
One convenient partition is presented in \cref{oscillate-region-fig}.
Give any curve emanating from one finger edge and terminating to infinity
and ensure that it has only one intersection point with the finger edges.
Let the initial point $(p,H)$ traverse the whole curve.
Then the parameters of other poles of the \PI~solutions will derive an infinite sequence of curves that separate the blank region into infinite ones; see \cref{oscillate-region-fig}.
Note that the fingers map to $-\infty$ in all the divisions in the upper blank region ($H>M$).
It is intriguing to note that dividing the parameter space into equivalent regions somewhat resembles the Brillouin zones in solid state physics.
Finally, it should be noted that, although the division is not unique, the asymptotic states of the division curves when $p, H\to-\infty$ seems to be stationary; see the coalescence of the red, blue and green dash-dotted curves in \cref{oscillate-region-fig}.
\begin{table}
  \centering
  \caption{Comparison of the asymptotic values and numerical values of $p$ and $H$ for the tritronqu\'{e}e solution}\label{table-comparison-pole-value}
  \begin{tabular}{c|rc|rc|cc}
  \toprule
\#	&   asymptotic 	&values  		& numerical 	& values 		& relative 		& errors  \\
\midrule
$n$	& $p_{n}\quad~$ & $H_{n}$		& $p_{n}\quad~$	& $H_{n}$		& $\Delta p_{n}/p_{n}$  		& $\Delta H_{n}/H_{n}$	  \\

 1	& $ 2.347\,592$ & $-0.063\,998$ & $ 2.384\,169$ & $-0.062\,139$ & $-0.015\,342$	& $0.029\,948$\\
 2	& $ 5.653\,529$ & $-0.239\,172$ & $ 5.664\,603$ & $-0.238\,306$ & $-0.001\,955$	& $0.003\,633$\\
 3	& $ 8.507\,435$ & $-0.441\,498$ & $ 8.513\,524$ & $-0.440\,920$ & $-0.000\,715$	& $0.001\,311$\\
 4	& $11.135\,278$ & $-0.661\,123$ & $11.139\,362$ & $-0.660\,688$ & $-0.000\,367$	& $0.000\,659$\\
 5	& $13.614\,968$ & $-0.893\,832$	& $13.617\,995$ & $-0.893\,476$ & $-0.000\,222$	& $0.000\,398$\\
 6	& $15.985\,888$ & $-1.137\,197$ & $15.988\,269$ & $-1.136\,886$ & $-0.000\,149$	& $0.000\,274$\\
 7	& $18.271\,630$ & $-1.389\,622$ & $18.273\,580$ & $-1.389\,352$ & $-0.000\,107$	& $0.000\,194$\\
 8	& $20.487\,814$ & $-1.649\,963$ & $20.489\,457$ & $-1.649\,712$ & $-0.000\,080$	& $0.000\,152$\\
 9	& $22.645\,492$ & $-1.917\,359$ & $22.646\,906$ & $-1.917\,144$ & $-0.000\,062$	& $0.000\,112$\\
10	& $24.752\,867$ & $-2.191\,134$ & $24.754\,104$ & $-2.190\,936$ & $-0.000\,050$	& $0.000\,090$\\
\bottomrule
  \end{tabular}
\end{table}

Another work of our numerical simulation is paid on the tritronqu\'{e}e solution of \PI~equation.
Using the initial values $y(0)\approx -0.187\,554\,308\,340\,494\,9$ and $y'(0)\approx 0.304\,905\,560\,261\,228\,9$ in \cite{Fornberg-Weideman},
we compute the values of $p_{n}$ (location of the $n$-th pole) and $H_{n}$ (the corresponding Laurent coefficient) for \PI's tritronqu\'{e}e solution.
It verifies the asymptotic formulas of $p_{n}$ and $H_{n}$ in \cref{cor-tritronquee}.
\cref{table-comparison-pole-value} gives a comparison of the numerical values and asymptotic values of $p_{n}$ and $H_{n}$ for this special solution.
Numerically, the relative errors can be described by
\begin{eqnarray*}
	\frac{\Delta p_n}{p_n} \sim -\frac{0.0045148}{\left(n-\frac{1}{2}\right)^2}, \qquad \frac{\Delta H_n}{H_n} \sim \frac{0.0081}{\left(n-\frac{1}{2}\right)^2}, \qquad n\to\infty.
\end{eqnarray*}

\section{Discussions}
\label{sec:discussion}

We have considered the connection problem of the first Painlev\'{e} transcendent between its poles and negative infinity theoretically and numerically.
In the theoretical aspect, we have derived the leading asymptotic behavior of Stokes multipliers for the reduced triconfluent Heun equation,
and then classified the real \PI~solutions asymptotically in terms of $p$ and $H$.
Some limiting-form connection formulas are also established.
As a by-product, we have obtained the large-$n$ asymptotics of $p_{n}$ and $H_{n}$ which characterize the $n$-th pole of the tritronqu\'{e}e solution of \PI.
By numerical simulations, we have drawn the phase diagram (\cref{oscillate-region-fig}) of the real \PI~solutions on the $(p,H)$ plane,
which verifies our main theoretical results (\cref{Fig-curves-Sigma-n}) and gives a complete classification of the real \PI~solutions with respect to $p$ and $H$.
It may be regarded as a partial numerical answer to Clarkson's open problem on \PI's connection formulas between its poles and negative infinity.
The following issues still need further investigation.
\begin{enumerate}
\item [(1)] In \cref{oscillate-region-fig}, we know the boundaries of the finger-shaped regions are the sets of points $(p,H)$ that give rise to the Type (B) solutions (separatrix solutions). We have only obtain the asymptotic equation of these curves theoretically, see \cref{eq-Sigma-n}. A natural problem is to derive the exact equation of them. If succeed, one may give a complete answer to Clarkson's open problem on the connection formulas of \PI~between poles and negative infinity.
\item [(2)] The connection problem of \PI~in the complex plane may be more challenging, while the corresponding results have more applications in mathematics and mathematical physics, especially for the tronqu\'{e}e and tritronqu\'{e}e solutions of \PI.
\item [(3)] Similar analysis can be done for the other Painlev\'{e} equations. The method of uniform asymptotic has been applied in the connection problems of the Painlev\'{e} equations between different singularities \cite{APC, Wong-Zhang-2009-PIII, Wong-Zhang-2009-PIV, Zeng-Zhao-2015}. We believe that this method works equally well to connecting the local behaviors between poles and negative (or positive) infinity of other Painlev\'e equations.
\item [(4)] The steepest descent approach for Riemann-Hilbert problems is an alternative and powerful tool to solve connection problems of Painlev\'{e} equations~\cite{BI-2012,Dai-Hu-2017}. It is also successful in deriving the pole distribution of the \PII~functions~\cite{Miller}.
    We suspect that this method is applicable to solve similar connection problems as in this paper,
    and it is one of the topics under further discussion.
\end{enumerate}

\

\appendix

\section{Proof of \texorpdfstring{\cref{lem-def-C0}}{Lemma 2.1}}
Since $\eta_{1}$ and $\eta_{2}$ are either both real or form a complex conjugate pair,
it is evident that $\Im\kappa^2=0$.
Moreover, we find that if $C(\xi)=-3/2^{2/3}$,
then $\eta_{1}=\eta_{2}$,
which implies that $\kappa^2=0$.
On the other hand, by a careful analysis, one may find that $K(C(\xi)):=\kappa^2$ is a decreasing function of $C(\xi)$.
In fact, we have
\[
K'(C(\xi))=\frac{1}{\pi i}\int_{\eta_{1}}^{\eta_{2}}\frac{s}{\sqrt{s^3+C(\xi)s+1}}ds.
\]
For any $C(\xi)\in\mathbb{R}$, by choosing the branch as described in \cref{branches},
we know that
\[
\arg\left(\frac{s}{\sqrt{s^3+C(\xi)s+1}}\right)\in\left(-\frac{\pi}{4},\frac{\pi}{4}\right)
\]
and $\arg{(ds)}\approx -\frac{\pi}{2}$.
Hence $K'(C(\xi))<0$ for all $C(\xi)\in\mathbb{R}$.
This proves the first part of \cref{eq-kappa-hat-kappa-sign}.

According to \cref{eq-def-kappa-hat-kappa}, we know that $\hat{K}(C(\xi)):=-\Im \hat{\kappa}^2$ is also a function of $C(\xi)$.
When $C(\xi)\leq 0$,
it can be derived from \cref{branches} that $\arg\left(\sqrt{s^3+C(\xi)s+1}\right)\in\left(\frac{3\pi}{4},\frac{13\pi}{12}\right)$
and $\arg{(ds)}\in \left(0,\frac{\pi}{6}\right)$.
It immediately follows that $\hat{K}(C(\xi)):=-\Im \hat{\kappa}^2<0$.

Next, we show that $\hat{K}(C(\xi))$ is increasing for $C(\xi)\in (0,+\infty)$ and $\hat{K}(+\infty)>0$.
When $C(\xi)\to+\infty$, we have $\arg\left(\sqrt{s^3+C(\xi)s+1}\right)\sim \frac{5\pi}{4}$ and $\arg{(ds)}\sim \frac{\pi}{2}$,
which implies $\arg{\hat{K}(C(\xi))}\sim\frac{9\pi}{4}$, and hence $\hat{K}(+\infty)>0$.
It remains to show
\begin{equation}\label{eq-hat-K-positive}
\hat{K}'(C(\xi))=-\Im\hat{\kappa}^2=\Re\left(\frac{1}{\pi}\int_{\eta_{0}}^{\eta_{1}}\frac{s}{\sqrt{s^3+C(\xi)s+1}}ds\right)>0.
\end{equation}
Set $\eta_{3}$ to be the intersection point of the above integral path with the imaginary axis. Then
\begin{equation}\label{eq-hat-K-two-parts}
\hat{K}'(C(\xi))=\Re\left(\frac{1}{\pi}\int_{\eta_{0}}^{\eta_{3}}\frac{s}{\sqrt{s^3+C(\xi)s+1}}ds\right) +\Re\left(\frac{1}{\pi}\int_{\eta_{3}}^{\eta_{1}}\frac{s}{\sqrt{s^3+C(\xi)s+1}}ds\right).
\end{equation}
From \cref{branches}, we have
\[
\arg\left(\frac{ds}{\sqrt{s^3+C(\xi)s+1}}\right)>-\frac{11\pi}{12}
\quad\text{and}\quad
\arg{s}>\frac{\pi}{2}.
\]
Therefore, the value of the first integral in \cref{eq-hat-K-two-parts} has a positive real part.
Similarly, the value of the second integral in \cref{eq-hat-K-two-parts} has a positive real part too,
since
\[
\arg\left(\frac{ds}{\sqrt{s^3+C(\xi)s+1}}\right)>-\frac{5\pi}{6}
\quad\text{and}\quad
\arg{s}>\frac{\pi}{3}
\]
for all $s$ on the corresponding integral path.
Hence, we have \cref{eq-hat-K-positive} hold for all $C(\xi)>0$,
which completes the proof of \cref{lem-def-C0}.

\section{Proof of \texorpdfstring{\cref{lemma-zeta-eta-infty-relation}}{Lemma 3.3}}
A simple calculation for the left-hand side of \cref{zeta-define} yields
\begin{equation}\label{zeta-define-left}
\begin{aligned}
&\int_{-\kappa}^{\zeta}(s^2-\kappa^2)^{\frac{1}{2}}ds=\int_{-\kappa}^{\kappa}(s^2-\kappa^2)^{\frac{1}{2}}ds+\int_{\kappa}^{\zeta}(s^2-\kappa^2)^{\frac{1}{2}}ds\\
=&\int_{-\kappa}^{\kappa}(s^2-\kappa^2)^{\frac{1}{2}}ds+\frac{1}{2}\left[\zeta(\zeta^2-\kappa^2)^{\frac{1}{2}}-\kappa^2\log(\zeta+(\zeta^2-\kappa^2)^{\frac{1}{2}})+\frac{\kappa^2}{2}\log{\kappa^2}\right]\\
=&\int_{-\kappa}^{\kappa}(s^2-\kappa^2)^{\frac{1}{2}}ds+\frac{1}{2}\zeta^2-\frac{\kappa^2}{2}\log\zeta-\frac{\kappa^2}{4}-\frac{\kappa^2}{2}\log{2}+\frac{\kappa^2}{4}\log{\kappa^2}+\mathcal{O}(\zeta^{-1})\end{aligned}
\end{equation}
as $\zeta\to\infty$. For the right-hand side of \cref{zeta-define}, we get
\begin{equation}\label{zeta-define-right}
\begin{aligned}
&\int_{\eta_{1}}^{\eta}F(s,\xi)^{\frac{1}{2}}ds=\int_{\eta_{1}}^{\eta_{2}}F(s,\xi)^{\frac{1}{2}}ds+\int_{\eta_{2}}^{\eta}F(s,\xi)^{\frac{1}{2}}ds\\
=&\int_{\eta_{1}}^{\eta_{2}}F(s,\xi)^{\frac{1}{2}}ds+\int_{\eta_{2}}^{\eta}\left[F(s,\xi)^{\frac{1}{2}}-2(s^{\frac{3}{2}}+\frac{C(\xi)}{2}s^{-\frac{1}{4}})\right]ds\\
&+\frac{4}{5}\eta^{\frac{5}{2}}+2C(\xi)\eta^{\frac{1}{2}}-\left(\frac{4}{5}(\eta_{2})^{\frac{5}{2}}+2C(\xi)(\eta_{2})^{\frac{1}{2}}\right).
\end{aligned}
\end{equation}
Note that $F(s,\xi)^{\frac{1}{2}}-2(s^{\frac{3}{2}}+\frac{C(\xi)}{2}s^{-\frac{1}{4}})=\mathcal{O}(s^{-\frac{3}{2}})$ as $s\to\infty$,
we can take the limit $\eta\to\infty$ in \cref{zeta-define-right}.
Comparing \cref{zeta-define-left} with \cref{zeta-define-right} and noting
that $\frac{1}{2}\zeta^2\sim\frac{4}{5}\eta^{\frac{5}{2}}$ as $\eta\to\infty$, we obtain \cref{eq-relation-zeta-eta} immediately.

\section*{Acknowledgements}
The authors would like to thank the two anonymous reviewers for helpful suggestions
and valuable comments that improve the manuscript signifcantly.
We also thank Professor Yu-Qiu Zhao for useful discussions.

\end{document}